\input{ifpdf.sty}
\pdffalse

\makeatletter

\newdimen\paperwidth
\newdimen\paperheight

\def\papersize#1#2{\let\p@persize\relax\paperwidth#1\paperheight#2}

\def\Afour{\papersize{210truemm}{297truemm}}

\let\p@persize\Afour

\let\onesidestyle\@twosidefalse
\let\twosidestyle\@twosidetrue

\def\margins{\@ifnextchar[{\@margins}{\@margins[\z@]}}

\def\@margins[#1]#2#3{
  \p@persize\dimen0 #3\dimen0 .5\dimen0\normalsize%
  \oddsidemargin-1truein\advance\oddsidemargin#2%
  \evensidemargin-1truein\advance\evensidemargin#2%
  \topmargin-1truein\advance\topmargin\dimen0\headsep\dimen0\footskip\dimen0%
  \textwidth\paperwidth\advance\textwidth-#2\advance\textwidth-#2%
  \textheight\paperheight\advance\textheight-#3\advance\textheight-#3%
  \headheight\baselineskip\advance\topmargin-.5\baselineskip%
  \advance\headsep-.5\baselineskip%
  \footheight\baselineskip
  \advance\textwidth-#1\advance\oddsidemargin#1
  \if@twoside\def\@themargin%
    {\ifodd\count\z@\oddsidemargin\else\evensidemargin\fi}\fi}

\def\headlinesep#1{\advance\topmargin\headsep\advance\topmargin -#1
  \advance\topmargin.5\baselineskip\headsep #1\advance\headsep-.5\baselineskip}

\def\headline{\if@twoside\let\n@xt\h@dlin@\else\let\n@xt\h@@dlin@\fi\n@xt}
  
\def\h@dlin@#1#2{%
  \def\@oddhead{%
    {{\leftskip\z@\rightskip\z@\noindent\normalsize#1}}}
  \def\@evenhead{%
    {{\leftskip\z@\rightskip\z@\noindent\normalsize#2}}}}

\def\h@@dlin@#1{%
  \def\@oddhead{{{\leftskip\z@\rightskip\z@\noindent\normalsize#1}}}}

\def\footline{\if@twoside\let\n@xt\f@tlin@\else\let\n@xt\f@@tlin@\fi\n@xt}

\def\f@tlin@#1#2{%
  \def\@oddfoot{%
    {{\leftskip\z@\rightskip\z@\noindent\normalsize#1}}}
  \def\@evenfoot{%
    {{\leftskip\z@\rightskip\z@\noindent\normalsize#2}}}}

\def\f@@tlin@#1{%
  \def\@oddfoot{{{\leftskip\z@\rightskip\z@\noindent\normalsize#1}}}}

\def\normalpage{\global\@specialpagefalse}

\makeatother

\makeatletter

\def\ft{\@ifnextchar[{\ft@s}{\ft@}}
\def\ft@{\ft@@@s[\f@size]}
\def\ft@s[{\@ifnextchar{a}{\ft@sz[}{\ft@@s[}}
\def\ft@@s[{\@ifnextchar{s}{\ft@sz[}{\ft@@@s[}}
\def\ft@@@s[#1]{\ft@sz[at #1pt]}
\def\ft@sz[#1]#2{\font\fonttemp=#2 #1\fonttemp\ignorespaces}

\def\pbf{\fontseries\bfdefault\selectfont\boldmath}

\makeatother

\makeatletter

\def\@@bold{bold}

\def\widebar{\ifx\math@version\@@bold
  \let\@widebar\@@@widebar\else\let\@widebar\@@widebar\fi\@widebar}

\def\@@widebar#1{\text{\setbox15\hbox{$#1$}%
  \dimen15 0.45\wd15\advance\dimen15 0.15\ht15%
  \dimen16\ht15\advance\dimen16 0.00em\advance\dimen16 0.3ex%
  \dimen17 0.65\wd15\advance\dimen17 0.05\ht15\advance\dimen17 0.1ex%
  \dimen18 0.035em\advance\dimen18 0.00ex
  \put[\dimen15,\dimen16][c]{\vrule depth 0pt height \dimen18 width \dimen17}}#1}

\def\@@@widebar#1{\text{\setbox15\hbox{$#1$}%
  \dimen15 0.45\wd15\advance\dimen15 0.15\ht15%
  \dimen16\ht15\advance\dimen16 0.00em\advance\dimen16 0.26ex%
  \dimen17 0.65\wd15\advance\dimen17 0.05\ht15\advance\dimen17 0.1ex%
  \dimen18 0.05em\advance\dimen18 0.00ex
  \put[\dimen15,\dimen16][c]{\vrule depth 0pt height \dimen18 width \dimen17}}#1}

\makeatother

\def\smallsquare{\raise-.065em\hbox{$\Box$}}
\def\smallblacksquare{%
  \kern.3ex\vrule depth-.03ex height1.27ex width1.15ex \kern-1.45ex \smallsquare}

\def\smallcircc{\mathop{\mkern3.5mu\text{\raise.58ex\hbox{\ft{lcircle10}a}}}}
\def\varemptyset{{\text{\raise.21ex\hbox{$\not$}}\mkern.15mu\mathrm{O}\mkern.15mu}}

  \let\epsilon\varepsilon
      \let\theta\vartheta
          \let\phi\varphi
   \let\emptyset\varemptyset

\documentstyle[12pt]{article}

\makeatletter

\let\Larg@\Large
\let\hug@\huge

\def\usepackage#1{\input{#1.sty}}

\input{geom.sty}
\input{option_keys.sty}

\let\input\@input

\def\r@adlabel#1#2{\global\@namedef{#1@\the\@key}{#2}}

\let\Large\Larg@
\let\huge\hug@

\def\smallskip{\vskip\smallskipamount}
\def\medskip{\vskip\medskipamount}
\def\bigskip{\vskip\bigskipamount}

\def\mytrivlist{\parsep\parskip\@nmbrlistfalse
  \my@trivlist \labelwidth\z@ \leftmargin\z@
  \itemindent\z@ \def\makelabel##1{##1}}

\def\my@trivlist{\global\@newlisttrue \@outerparskip\parskip}

\def\end#1{\csname end#1\endcsname\@checkend{#1}%
  \expandafter\endgroup\if@endpe\@doendpe\fi
  \if@ignore \global\@ignorefalse \ignorespaces\fi}

\input{epsf.sty}
\usepackage{graphicx}

\def\put{\@ifnextchar[{\@put}{\@@rput[\z@,\z@][r]}}
\def\@put[#1]{\@ifnextchar[{\@@put[#1]}{\@@@@@put[#1]}}
\def\@@put[#1][{\@ifnextchar{l}{\@@lput[#1][}{\@@@put[#1][}}
\def\@@@put[#1][{\@ifnextchar{c}{\@@cput[#1][}{\@@@@put[#1][}}
\def\@@@@put[#1][{\@ifnextchar{r}{\@@rput[#1][}{\relax}}
\def\@@@@@put[{\@ifnextchar{l}{\@@lput[\z@,\z@][}{\@@@@@@put[}}
\def\@@@@@@put[{\@ifnextchar{c}{\@@cput[\z@,\z@][}{\@@@@@@@put[}}
\def\@@@@@@@put[{\@ifnextchar{r}{\@@rput[\z@,\z@][}{\@@@@@@@@put[}}
\def\@@@@@@@@put[#1]{\@@rput[#1][r]}

\let\hm@d@\leavevmode

\long\def\@@lput[#1,#2][l]#3{\setbox0\hbox{#3}\hm@d@\raise#2\hbox to\z@{\dimen0 #1%
  \advance\dimen0-\wd0\kern\dimen0\dp0\z@\ht0\z@\wd0\z@\box0\hss}\ignorespaces}
\long\def\@@cput[#1,#2][c]#3{\setbox0\hbox{#3}\hm@d@\raise#2\hbox to\z@{\dimen0 #1%
  \advance\dimen0-.5\wd0\kern\dimen0\dp0\z@\ht0\z@\wd0\z@\box0\hss}\ignorespaces}
\long\def\@@rput[#1,#2][r]#3{\setbox0\hbox{\kern#1\raise#2\hbox{#3}}%
  \dp0\z@\ht0\z@\wd0\z@\hm@d@\box0\ignorespaces}

\def\flbox{\@ifnextchar[{\@flbox}{\@@rflbox[\z@,\z@][r]}}
\def\@flbox[#1]{\@ifnextchar[{\@@flbox[#1]}{\@@@@@flbox[#1]}}
\def\@@flbox[#1][{\@ifnextchar{l}{\@@lflbox[#1][}{\@@@flbox[#1][}}
\def\@@@flbox[#1][{\@ifnextchar{c}{\@@cflbox[#1][}{\@@@@flbox[#1][}}
\def\@@@@flbox[#1][{\@ifnextchar{r}{\@@rflbox[#1][}{\relax}}
\def\@@@@@flbox[{\@ifnextchar{l}{\@@lflbox[\z@,\z@][}{\@@@@@@flbox[}}
\def\@@@@@@flbox[{\@ifnextchar{c}{\@@cflbox[\z@,\z@][}{\@@@@@@@flbox[}}
\def\@@@@@@@flbox[{\@ifnextchar{r}{\@@rflbox[\z@,\z@][}{\@@@@@@@@flbox[}}
\def\@@@@@@@@flbox[#1]{\@@rflbox[#1][r]}
\long\def\@@lflbox[#1,#2][l]#3{\@@lput[#1,#2][l]{%
  \vtop{\leftskip\z@\parindent\z@\raggedleft\hm@d@#3}}}
\long\def\@@cflbox[#1,#2][c]#3{\@@cput[#1,#2][c]{%
  \vtop{\leftskip\z@\parindent\z@\raggedcenter\hm@d@#3}}}
\long\def\@@rflbox[#1,#2][r]#3{\@@rput[#1,#2][r]{%
  \vtop{\leftskip\z@\parindent\z@\raggedright\hm@d@#3}}}

\def\maketitle{\par
 \begingroup
 \def\thefootnote{\fnsymbol{footnote}}
 \def\@makefnmark{\hbox 
 to 0pt{$^{\@thefnmark}$\hss}} 
 \if@twocolumn 
 \twocolumn[\@maketitle] 
 \else 
 \global\@topnum\z@ \@maketitle \fi\thispagestyle{plain}\@thanks
 \endgroup
 \setcounter{footnote}{0}
 \let\maketitle\relax
 \let\@maketitle\relax
 \gdef\@thanks{}\gdef\@author{}\gdef\@title{}\let\thanks\relax}

\def\@maketitle{ 
 \null
 \vskip 2em \begin{center}
 {\LARGE \@title \par} \vskip 1.5em {\large \lineskip .5em
\begin{tabular}[t]{c}\@author 
 \end{tabular}\par} 
 \vskip 1em {\large \@date} \end{center}
 \par
 \vskip 1.5em}
 
\def\partbeforeskip#1{\def\p@rtbeforeskip{#1}}
\def\partstyle#1{\def\p@rtstyl@{#1}}
\def\partdot#1{\def\partd@t{#1}}
\def\partafterskip#1{\def\p@rtafterskip{#1}}
\def\partintrostyle#1{\def\partintr@styl@{#1}}
\def\partintrodot#1{\def\partintr@dot{#1}}
\long\def\partintrosep#1{\long\def\partintr@sep{#1}}
\def\partnewpagetrue{\def\p@rtnewp@ge{\newpage}}
\def\partnewpagefalse{\long\def\p@rtnewp@ge{\par}}

\partbeforeskip{4ex}
\partstyle{\centering\Large\bf}
\partdot{}
\partafterskip{3ex}
\partintrostyle{\large}
\partintrodot{}
\partintrosep{\par}
\partnewpagefalse

\def\partname{Part}
\def\part{\p@rtnewp@ge\addvspace\p@rtbeforeskip\@afterindentfalse\secdef\@part\@spart}

\def\@part[#1]#2{\ifnum \c@secnumdepth >-1\relax  
        \refstepcounter{part}                     
        \def\@tempa{\addcontentsline{toc}{part}}  %
        \expandafter\@tempa\expandafter{\thepart  
          \hspace{1em}#1}\else                    
        \addcontentsline{toc}{part}{#1}\fi        
   {\p@rtstyl@                       
    \ifnum \c@secnumdepth >-1\relax        
      {\partintr@styl@\partname\ \thepart  
       \partintr@dot}\partintr@sep\nobreak 
    \fi                                    
    #2\partd@t\markboth{}{}\par}
    \nobreak                       
    \vskip\p@rtafterskip           
   \@afterheading                  
    }                              

\def\@spart#1{{\p@rtcentering\p@rtstyl@                      
    #1\partd@t\par}                 
    \nobreak                        
    \vskip\p@rtafterskip            
    \@afterheading                  
  }                                 

\newif\ifsection@ftind
\newif\ifsection@ftpar

\def\sectionbeforeskip#1{\def\s@ctbeforeskip{#1}}
\def\sectionstyle#1{\def\s@ctstyl@{#1}}
\def\sectiondot#1{\def\sectiond@t{#1}}
\def\sectionafterskip#1{\def\s@ctafterskip{#1}}
\def\sectionintrostyle#1{\def\sectionintr@styl@{#1}}
\def\sectionintro#1{\def\sectionintr@{#1}}
\def\sectionintrodot#1{\def\sectionintr@dot{#1}}
\def\sectionintrosep#1{\def\sectionintr@sep{#1}}
\def\sectionindenttrue{\def\s@ctind{\parindent}}
\def\sectionindentfalse{\def\s@ctind{\z@}}
\def\sectionafterindenttrue{\section@ftindtrue}
\def\sectionafterindentfalse{\section@ftindfalse}
\def\sectionafternewlinetrue{\section@ftpartrue}
\def\sectionafternewlinefalse{\section@ftparfalse}

\newif\ifsubsection@ftind
\newif\ifsubsection@ftpar

\def\subsectionbeforeskip#1{\def\ss@ctbeforeskip{#1}}
\def\subsectionstyle#1{\def\ss@ctstyl@{#1}}
\def\subsectiondot#1{\def\subsectiond@t{#1}}
\def\subsectionafterskip#1{\def\ss@ctafterskip{#1}}
\def\subsectionintrostyle#1{\def\subsectionintr@styl@{#1}}
\def\subsectionintro#1{\def\subsectionintr@{#1}}
\def\subsectionintrodot#1{\def\subsectionintr@dot{#1}}
\def\subsectionintrosep#1{\def\subsectionintr@sep{#1}}
\def\subsectionindenttrue{\def\ss@ctind{\parindent}}
\def\subsectionindentfalse{\def\ss@ctind{\z@}}
\def\subsectionafterindenttrue{\subsection@ftindtrue}
\def\subsectionafterindentfalse{\subsection@ftindfalse}
\def\subsectionafternewlinetrue{\subsection@ftpartrue}
\def\subsectionafternewlinefalse{\subsection@ftparfalse}

\newif\ifsubsubsection@ftind
\newif\ifsubsubsection@ftpar

\def\subsubsectionbeforeskip#1{\def\sss@ctbeforeskip{#1}}
\def\subsubsectionstyle#1{\def\sss@ctstyl@{#1}}
\def\subsubsectiondot#1{\def\subsubsectiond@t{#1}}
\def\subsubsectionafterskip#1{\def\sss@ctafterskip{#1}}
\def\subsubsectionintrostyle#1{\def\subsubsectionintr@styl@{#1}}
\def\subsubsectionintro#1{\def\subsubsectionintr@{#1}}
\def\subsubsectionintrodot#1{\def\subsubsectionintr@dot{#1}}
\def\subsubsectionintrosep#1{\def\subsubsectionintr@sep{#1}}
\def\subsubsectionindenttrue{\def\sss@ctind{\parindent}}
\def\subsubsectionindentfalse{\def\sss@ctind{\z@}}
\def\subsubsectionafterindenttrue{\subsubsection@ftindtrue}
\def\subsubsectionafterindentfalse{\subsubsection@ftindfalse}
\def\subsubsectionafternewlinetrue{\subsubsection@ftpartrue}
\def\subsubsectionafternewlinefalse{\subsubsection@ftparfalse}

\newif\ifparagraph@ftind
\newif\ifparagraph@ftpar

\def\paragraphbeforeskip#1{\def\p@rbeforeskip{#1}}
\def\paragraphstyle#1{\def\p@rstyl@{#1}}
\def\paragraphdot#1{\def\paragraphd@t{#1}}
\def\paragraphafterskip#1{\def\p@rafterskip{#1}}
\def\paragraphintrostyle#1{\def\paragraphintr@styl@{#1}}
\def\paragraphintro#1{\def\paragraphintr@{#1}}
\def\paragraphintrodot#1{\def\paragraphintr@dot{#1}}
\def\paragraphintrosep#1{\def\paragraphintr@sep{#1}}
\def\paragraphindenttrue{\def\p@rind{\parindent}}
\def\paragraphindentfalse{\def\p@rind{\z@}}
\def\paragraphafterindenttrue{\paragraph@ftindtrue}
\def\paragraphafterindentfalse{\paragraph@ftindfalse}
\def\paragraphafternewlinetrue{\paragraph@ftpartrue}
\def\paragraphafternewlinefalse{\paragraph@ftparfalse}

\newif\ifsubparagraph@ftind
\newif\ifsubparagraph@ftpar

\def\subparagraphbeforeskip#1{\def\sp@rbeforeskip{#1}}
\def\subparagraphstyle#1{\def\sp@rstyl@{#1}}
\def\subparagraphdot#1{\def\subparagraphd@t{#1}}
\def\subparagraphafterskip#1{\def\sp@rafterskip{#1}}
\def\subparagraphintrostyle#1{\def\subparagraphintr@styl@{#1}}
\def\subparagraphintro#1{\def\subparagraphintr@{#1}}
\def\subparagraphintrodot#1{\def\subparagraphintr@dot{#1}}
\def\subparagraphintrosep#1{\def\subparagraphintr@sep{#1}}
\def\subparagraphindenttrue{\def\sp@rind{\parindent}}
\def\subparagraphindentfalse{\def\sp@rind{\z@}}
\def\subparagraphafterindenttrue{\subparagraph@ftindtrue}
\def\subparagraphafterindentfalse{\subparagraph@ftindfalse}
\def\subparagraphafternewlinetrue{\subparagraph@ftpartrue}
\def\subparagraphafternewlinefalse{\subparagraph@ftparfalse}

\sectionbeforeskip{\bigskipamount}
\sectionstyle{\large\bf}
\sectiondot{}
\sectionafterskip{.5\bigskipamount}
\sectionintrostyle{}
\sectionintro{}
\sectionintrodot{.}
\sectionintrosep{1.25ex}
\sectionindentfalse
\sectionafterindenttrue
\sectionafternewlinetrue

\subsectionbeforeskip{.8\bigskipamount}
\subsectionstyle{\normalsize\bf}
\subsectiondot{}
\subsectionafterskip{.4\bigskipamount}
\subsectionintrostyle{}
\subsectionintro{}
\subsectionintrodot{.}
\subsectionintrosep{1.25ex}
\subsectionindentfalse
\subsectionafterindenttrue
\subsectionafternewlinetrue

\subsubsectionbeforeskip{.6\bigskipamount}
\subsubsectionstyle{\normalsize\bf}
\subsubsectiondot{}
\subsubsectionafterskip{.3\bigskipamount}
\subsubsectionintrostyle{}
\subsubsectionintro{}
\subsubsectionintrodot{.}
\subsubsectionintrosep{1.25ex}
\subsubsectionindentfalse
\subsubsectionafterindenttrue
\subsubsectionafternewlinetrue

\paragraphbeforeskip{.5\bigskipamount}
\paragraphstyle{\normalsize\bf}
\paragraphdot{.}
\paragraphafterskip{1.25ex}
\paragraphintrostyle{}
\paragraphintro{}
\paragraphintrodot{.}
\paragraphintrosep{1.25ex}
\paragraphindentfalse
\paragraphafterindenttrue
\paragraphafternewlinefalse

\subparagraphbeforeskip{.5\bigskipamount}
\subparagraphstyle{\normalsize\bf}
\subparagraphdot{.}
\subparagraphafterskip{1.25ex}
\subparagraphintrostyle{}
\subparagraphintro{}
\subparagraphintrodot{.}
\subparagraphintrosep{1.25ex}
\subparagraphindenttrue
\subparagraphafterindenttrue
\subparagraphafternewlinefalse

\let\@partoken\par
\long\def\@@gobble#1{}
\def\ignorepar{\@ifnextchar\@partoken{\expandafter\ignorepar\@@gobble}{\ignorespaces}}

\def\@startsection#1#2#3#4#5#6{
   \@tempskipa #4\relax
   \csname if#1@ftind\endcsname\@afterindenttrue\else\@afterindentfalse\fi
   \advance\@tempskipa by\presection
   \if@nobreak \everypar{}\else
     \addpenalty{\@secpenalty}\addvspace{\@tempskipa}%
     \allowbreak\vskip -\presection \fi \@ifstar
     {\@ssect{#1}{#2}{#3}{#4}{#5}{#6}}{\@dblarg{\@sect{#1}{#2}{#3}{#4}{#5}{#6}}}}

\def\@sect#1#2#3#4#5#6[#7]#8{\def\object@type{#1}%
   \ifnum #2>\c@secnumdepth\def\@svsec{}\def\@tempb{}%
      \else\refstepcounter{#1}\def\@svsec{{\csname #1intr@styl@\endcsname%
        {\csname #1intr@\endcsname}\csname the#1\endcsname%
        \csname #1intr@dot\endcsname\kern\csname #1intr@sep\endcsname}}%
        \edef\@tempb{\noexpand\numberline{\csname the#1\endcsname}}\fi%
   \def\@tempa{\addcontentsline{toc}{#1}}%
   \csname if#1@ftpar\endcsname%
      \begingroup #6\relax%
        \@hangfrom{\hskip #3\relax\@svsec}{\interlinepenalty \@M{#8}%
        \csname #1d@t\endcsname\par}%
      \endgroup%
      \csname #1mark\endcsname{#7}%
      \expandafter\@tempa\expandafter{\@tempb #7}%
      \ifautolabel\label*{#8}\fi%
   \else%
      \def\@svsechd{#6\hskip #3\relax%
         \@svsec{#8}%
         \csname #1d@t\endcsname%
         \csname #1mark\endcsname{#7}%
         \expandafter\@tempa\expandafter{\@tempb #7}%
         \ifautolabel\label*{#8}\fi}\fi%
   \@xsect{#1}{#5}\ignorepar}

\def\@ssect#1#2#3#4#5#6#7{%
   \ifnum #2>\c@secnumdepth\def\@tempb{}\else \def\@tempb{\numberline{}}\fi%
     \def\@tempa{\addcontentsline{toc}{s#1}}%
     \csname if#1@ftpar\endcsname
        \begingroup #6\relax
           \@hangfrom{\hskip #3}{\interlinepenalty \@M{#7}%
           \csname #1d@t\endcsname\par}%
        \endgroup
        \csname s#1mark\endcsname{#7}%
        \ifstarredcontents\expandafter\@tempa\expandafter{\@tempb #7}\fi%
        \ifautolabel\label*{#7}\fi%
     \else%
        \def\@svsechd{#6\hskip #3\relax{#7}%
        \csname #1d@t\endcsname%
        \csname s#1mark\endcsname{#7}%
        \ifautolabel\label*{#7}\fi}\fi
   \@xsect{#1}{#5}\ignorepar}

\def\@xsect#1#2{
   \csname if#1@ftpar\endcsname 
       \par \nobreak \vskip #2\relax \@afterheading
    \else \global\@nobreakfalse \global\@noskipsectrue
       \everypar{\if@noskipsec \global\@noskipsecfalse
                   \clubpenalty\@M \hskip -\parindent
                   \begingroup \@svsechd \endgroup \unskip
                   \hskip #2\relax  
                  \else \clubpenalty \@clubpenalty
                    \everypar{}\fi}\fi\ignorespaces}

\def\section{\@startsection{section}{1}{\s@ctind}
  {\s@ctbeforeskip}{\s@ctafterskip}{\s@ctstyl@}}
\def\subsection{\@startsection{subsection}{2}{\ss@ctind}
  {\ss@ctbeforeskip}{\ss@ctafterskip}{\ss@ctstyl@}}
\def\subsubsection{\@startsection{subsubsection}{3}{\sss@ctind}
  {\sss@ctbeforeskip}{\sss@ctafterskip}{\sss@ctstyl@}}
\def\paragraph{\@startsection{paragraph}{4}{\p@rind}
  {\p@rbeforeskip}{\p@rafterskip}{\p@rstyl@}}
\def\subparagraph{\@startsection{subparagraph}{4}{\sp@rind}
  {\sp@rbeforeskip}{\sp@rafterskip}{\sp@rstyl@}}

\def\statementabove#1{\def\th@bove{#1}}
\def\statementstyle#1{\def\thstyl@{#1}}
\def\statementbelow#1{\def\thb@low{#1}}
\def\statementindentfalse{\let\thind@nt\relax}
\def\statementindenttrue{\let\thind@nt\indent}

\def\statementintrostyle#1{\def\thintr@style{#1}}
\def\statementintrodot#1{\def\thintr@dot{#1}}
\def\statementintrosep#1{\def\thintr@sep{#1}}
\def\statementintrobrackets#1#2{\def\thintr@left{#1}\def\thintr@right{#2}}

\statementabove{\medskip}
\statementstyle{\sl}
\statementbelow{\medskip}
\statementindenttrue

\statementintrostyle{\normalshape\bf}
\statementintrodot{.}
\statementintrosep{\kern1.25ex}
\statementintrobrackets{(}{)}

\def\@thskip{\dimen100\lastskip\vskip-\dimen100%
  \th@bove\dimen101\lastskip\vskip-\dimen101%
  \ifdim\dimen100>\dimen101\else\dimen100\dimen101\fi\vskip\dimen100\vskip0pt}

\long\def\@@newtheorem#1#2#3{%
  \newenvironment{#3}%
    {\def\object@type{#3}\par\@thskip%
     \@ifnextchar[{\@enva{#3}{\thstyl@#1{#2}}}{\@envb{#3}{\thstyl@#1{#2}}}}%
    {\end{#3@}}%
  \@ifnextchar[{\@nothm{#3}}{\@nnthm{#3}}}

\def\@nothm#1[#2]#3{%
  \@ifundefined{c@#2}{\@latexerr{No theorem environment `#2' defined}\@eha}%
  {\expandafter\@ifdefinable\csname #1@\endcsname
  {\global\@namedef{the#1}{\@nameuse{the#2}}%
   \global\@namedef{c@#1}{\@nameuse{c@#2}}
   \global\@namedef{p@#1}{\@nameuse{p@#2}}
   \global\@namedef{#1@}{\@nnnthm{#2}{#3}}%
   \global\@namedef{end#1@}{\@endtheorem}}}}

\def\@nnnthm#1#2{\refstepcounter
    {#1}\@ifnextchar[{\@ynnnthm{#1}{#2}}{\@xnnnthm{#1}{#2}}}
 
\def\@xnnnthm#1#2{\@begintheorem{#2}{\csname the#1\endcsname}\ignorespaces}
\def\@ynnnthm#1#2[#3]{\@opargbegintheorem{#2}{\csname the#1\endcsname}{#3}\ignorespaces}

\def\renewtheorem{\@ifnextchar[{\@renewtheorem}{\@renewtheorem[{}{}]}}

\long\def\@renewtheorem[#1]{\@@renewtheorem#1}

\long\def\@@renewtheorem#1#2#3{%
  \expandafter\let\csname#3@\endcsname\undefined
  \renewenvironment{#3}%
    {\def\object@type{#3}\par\@thskip%
     \@ifnextchar[{\@enva{#3}{\thstyl@#1{#2}}}{\@envb{#3}{\thstyl@#1{#2}}}}%
    {\end{#3@}}%
  \@ifnextchar[{\@nothm{#3}}{\@nnthm{#3}}}

\def\@begintheorem#1#2{\@opargbegintheorem{#1}{#2}{}}

\def\@opargbegintheorem#1#2#3{%
        \edef\@tempx{#1}%
        \expandafter\let\expandafter\@tempy#2
        \def\@tempz{#3}%
        \mytrivlist\item[\thind@nt\hskip\labelsep%
        {\thintr@style%
          #1\ifx\@tempx\@empty\else\ifx\@tempy\relax\else\kern1ex\fi\fi#2%
          \ifx\@tempz\@empty%
            \ifx\@tempx\@empty\ifx\@tempy\relax%
            \else\thintr@dot\thintr@sep\fi\else\thintr@dot\thintr@sep\fi%
            \else%
            \ifx\@tempx\@empty\ifx\@tempy\relax\else\kern1ex\fi\else\kern1ex\fi%
           \thintr@left{#3}\thintr@right\thintr@dot\thintr@sep\fi}%
            \hskip-\labelsep]%
        \ifautolabel\label*{#3}\fi}

\def\@endtheorem{\endtrivlist\thb@low}

\def\proofname{Proof}

\def\proofabove#1{\def\pf@bove{#1}}
\def\proofstyle#1{\def\pfstyl@{#1}}
\def\proofbelow#1{\def\pfb@low{#1}}
\def\proofindentfalse{\let\pfind@nt\relax}
\def\proofindenttrue{\let\pfind@nt\indent}

\def\proofintrostyle#1{\def\pfintr@style{#1}}
\def\proofintrodot#1{\def\pfintr@dot{#1}}
\def\proofintrosep#1{\def\pfintr@sep{#1}}
\def\proofintrobrackets#1#2{\def\pfintr@left{#1}\def\pfintr@right{#2}}

\proofabove{\medskip}
\proofstyle{}
\proofbelow{\medskip}
\proofindenttrue

\proofintrostyle{\sl}
\proofintrodot{.}
\proofintrosep{\kern1.25ex}
\proofintrobrackets{of\kern1ex}{}

\def\@pfskip{\dimen100\lastskip\vskip-\dimen100%
  \pf@bove\dimen101\lastskip\vskip-\dimen101%
  \ifdim\dimen100>\dimen101\else\dimen100\dimen101\fi\vskip\dimen100\vskip0pt}

\renewenvironment{proof}%
  {\@pfskip\mytrivlist\item[\pfind@nt]\@ifnextchar[{\pro@f}{\pro@f[\prooftag]}}
  {\ifvoid\provedbox\else\hproved\fi\endtrivlist\pfb@low}

\def\pro@f[#1]{\setbox\provedbox\hbox{\provedboxcontents{#1}}\proofintro{#1}}

\def\proofintro#1{\expandafter\def\expandafter\@tempa\expandafter{#1}%
  {\pfintr@style{\proofname\ifx\@tempa\empty\else\kern1ex\pfintr@left{#1}%
  \pfintr@right\fi}\pfintr@dot\pfintr@sep}\pfstyl@\ignorespaces}

\def\provedmark#1{\def\prm@rk{#1}}
\def\provedsep#1{\def\prs@p{#1}}

\provedmark{$\square$}
\provedsep{\kern1.25ex}

\def\provedtexttrue{\def\prb@x##1{\fbox{\small##1}}}
\def\provedtextfalse{\def\prb@x##1{\prm@rk}}
\def\provedmarkrighttrue{\let\prhf@l\hfill}
\def\provedmarkrightfalse{\let\prhf@l\relax}

\provedtextfalse
\provedmarkrightfalse

\def\provedboxcontents#1{\expandafter\def\expandafter\@tempa\expandafter{#1}%
  \ifx\@tempa\empty\prm@rk\else\prb@x{#1}\fi}

\def\proved{\ifmmode\eqno{\box\provedbox}\else\hproved\fi}

\def\hproved{\unskip\nobreak\prhf@l\penalty50\prs@p\hbox{}\nobreak\prhf@l
  \box\provedbox{\finalhyphendemerits=0\par}}

\def\captionstyle#1{\def\c@ptstyl@{#1}}
\def\captionintrostyle#1{\def\c@pintr@style{#1}}
\def\captionintrodot#1{\def\c@pintr@dot{#1}}
\def\captionintrosep#1{\def\c@pintr@sep{#1}}

\captionstyle{\small\sf}
\captionintrostyle{\bf}
\captionintrodot{.}
\captionintrosep{\hskip1.25ex}

\long\def\@makecaption#1#2{%
    \vskip\captionskip
    \setbox\@tempboxa\hbox{%
      \ifproofing\@ifundefined{the@label}{}
        {\hbox to 0pt{\vbox to 0pt{\vss\hbox{\tiny\the@label}\bigskip}\hss}}\fi
      \c@ptstyl@{\c@pintr@style #1\c@pintr@dot}\ignorespaces #2}%
    \@captionwidth=\hsize \advance\@captionwidth-2\@captionmargin
    \ifdim \wd\@tempboxa >\@captionwidth {%
        \rightskip=\@captionmargin\leftskip=\@captionmargin
        \unhbox\@tempboxa\par}%
      \else
        \hbox to\hsize{\hfil\box\@tempboxa\hfil}%
    \fi}

\def\end@Float#1{%
  \expandafter\caption\expandafter[\the@title]{%
   {\c@pintr@style%
   \ifx\the@caption\empty\ifx\the@title\empty
   \else\c@pintr@sep\fi\else\c@pintr@sep\fi
    \the@title\ifx\the@caption\empty%
     \expandafter\label\expandafter*\expandafter{\the@label}%
    \else\ifx\the@title\empty%
     \expandafter\label\expandafter*\expandafter{\the@label}%
    \else\c@pintr@dot\c@pintr@sep%
     \expandafter\label\expandafter*\expandafter{\the@label}\fi\fi}%
   \ignorespaces\the@caption}%
  \end{#1}}

\renewenvironment{Figure}{\@ifnextchar[%
  {\@myFloat{figure}}{\@myFloat{figure}[htbp]}}{\end@Float{figure}}

\def\@myFloat#1[#2]#3{%
  \def\color@hbox{}\def\color@vbox{}\def\color@endbox{}%
  \begin{#1}[#2]\def\the@label{#3}}

\def\fig#1{\@ifnextchar[{\@fig{#1}}{\@fig{#1}[0pt]}}

\def\@fig#1[#2]#3{\@ifnextchar[{\@@fig{#1}[#2]{#3}}{\@@fig{#1}[#2]{#3}[0pt]}}
\def\@@fig#1[#2]#3[#4]#5#6{%
  \def\the@title{#5}\def\the@caption{#6}\centerline{\fig@{#1}{#2}{#3}}\vskip#4}

\def\fig@@#1#2#3{\leavevmode{\figstyl@\vrule width 0pt height 1.8ex%
 \smash{\framebox{\strut\def\@temp{#1}\ifx\@temp\@empty{ #3 }\else{ #1 }\fi}}}}
\def\fig@@@#1#2#3{\leavevmode\kern#2\epsfbox{#3}}

\def\figstyle#1{\def\figstyl@{#1}}

\figstyle{\fontfamily{cmr}\fontshape{n}\fontseries{m}\selectfont\normalsize}

\newcounter{diagram}

\let\thediagram\theequation
\def\ftype@diagram{2}
\def\ext@diagram{lod}
\def\diagram{\@float{diagram}}
\let\enddiagram\end@float
\newif\if@diagnum

\def\diag#1{\@ifnextchar[{\@diag{#1}}{\@diag{#1}[0pt]}}

\def\@diag#1[#2]#3{\@ifnextchar[{\@@diag{#1}[#2]{#3}}{\@@diag{#1}[#2]{#3}[0pt]}}

\def\@@diag#1[#2]#3[#4]#5{
  \def\the@tag{#5}\@eqnswtrue%
  \centerline{\setbox0\hbox{\diag@{#1}{#2}{#3}}
  \dimen0 -0.5\wd0\dimen1 0.5\ht0\box0%
  \advance\dimen0 0.5\hsize\advance\dimen0 -\rightskip\advance\dimen1 #4%
  \let\@currentlabel\the@tag%
  \setbox0\hbox to 0pt{\hss%
    \fontfamily{cmr}\fontshape{n}\fontseries{m}\selectfont(\the@tag)}%
  \ifx\the@tag\@empty\refstepcounter{equation}\let\@currentlabel\theequation%
    \setbox0\hbox to 0pt{\hss%
      \fontfamily{cmr}\fontshape{n}\fontseries{m}\selectfont(\thediagram)}\fi%
  \if@eqnsw\else\let\@currentlabel\relax\setbox0\hbox to 0pt{}\fi%
  \advance\dimen1 -0.5\ht0%
  \put[\dimen0,\dimen1][l]{%
    \box0\expandafter\label\expandafter*\expandafter{\the@label}\kern0.15em}}}

\def\diag@@#1#2#3{\leavevmode{\diagstyl@\vrule width 0pt height 1.8ex%
 \smash{\framebox{\strut\def\@temp{#1}\ifx\@temp\@empty{ #3 }\else{ #1 }\fi}}}}
\def\diag@@@#1#2#3{\leavevmode\kern#2\epsfbox{#3}}

\def\diagstyle#1{\def\diagstyl@{#1}}

\diagstyle{\fontfamily{cmr}\fontshape{n}\fontseries{m}\selectfont\normalsize}

\def\showfiguresfalse{\let\fig@\fig@@}
\def\showfigurestrue{\let\fig@\fig@@@}

\def\showdiagramsfalse{\let\diag@\diag@@}
\def\showdiagramstrue{\let\diag@\diag@@@}

\showfigurestrue
\showdiagramstrue

\def\n@number{\@eqnswfalse\let\@currentlabel\relax\let\the@tag\relax}

\def\equation{$$
  \@eqnswtrue\def\object@type{equation}\let\nonumber\n@number%
  \advance\c@equation1\edef\@currentlabel{\theequation}\advance\c@equation-1%
  \def\the@tag{\refstepcounter{equation}\eqno\hbox{\@eqnnum}}}

\def\tag#1{\edef\@currentlabel{#1}\def\the@tag{\eqno\hbox{\reset@font\rm(#1)}}}

\def\endequation{\the@tag$$
  \global\@ignoretrue}

\let\it@m\item
\def\item{\@ifnextchar[{\item@}{\item@@}}
\def\item@[#1]{\it@m[#1]\vskip-\lastskip\vskip\itemsep}
\def\item@@{\it@m\vskip-\lastskip\vskip\itemsep}
\def\s@titemsep{\@ifnextchar[{\s@@titemsep}{\relax}}
\def\s@@titemsep[#1]{\itemsep#1}

\let\@itemize\itemize
\let\@enditemize\enditemize
\renewenvironment{itemize}
{\@itemize\itemsep3pt\parsep0pt\topsep0pt\partopsep0pt\s@titemsep}
{\@enditemize\vskip-\lastskip\vskip\itemsep}

\let\@enumerate\enumerate
\let\@endenumerate\endenumerate
\renewenvironment{enumerate}
{\@enumerate\itemsep3pt\parsep0pt\topsep0pt\partopsep0pt\s@titemsep}
{\@endenumerate\vskip-\lastskip\vskip\itemsep}

\let\@description\description
\let\@enddescription\enddescription
\renewenvironment{description}
{\@description\itemsep3pt\parsep0pt\topsep0pt\partopsep0pt\s@titemsep}
{\@enddescription\vskip-\lastskip\vskip\itemsep}

\topsep\dimen200
\partopsep\dimen201

\@definecounter{bibenumi}

\def\thebibliography#1{%
 \section*{\refname}\vskip-\lastskip%
 \list{[\arabic{bibenumi}]}{\topsep0pt\settowidth\labelwidth{[#1]}%
 \leftmargin\labelwidth\advance\leftmargin\labelsep\usecounter{bibenumi}}%
 \def\newblock{\hskip .11em plus .33em minus .07em}%
 \sloppy\clubpenalty4000\widowpenalty4000\sfcode`\.=1000\relax}

\let\@ref@\ref
\let\@pageref@\pageref
\let\@fullref@\fullref
\let\@Fullref@\Fullref
\let\@reftype@\reftype
\let\@Reftype@\Reftype
\let\@label@\label
\let\@cite@\cite
\let\@bibitem@\bibitem

\def\label{\@ifnextchar*{\label@}{\label@{}}}
\def\label@#1#2{\@label@#1{#2}\putl@bel{#2}\ignorespaces}
\def\putl@b@l#1{\put[0pt,.25\baselineskip]{%
  \hbox{\labc@lor{\fontfamily{cmr}\fontshape{n}\fontseries{m}\selectfont%
  \tiny\setbox5\hbox{\vphantom{X}\smash{\ns#1}}%
  \hbox to 0pt{\hss\tiny$\blacktriangledown$\kern-.085em}%
  \raise2.25ex\hbox to 0pt{\hss\framebox{\box5}}}}}}

\def\putr@fl@bel#1{{\let\labc@lor\refc@lor\putl@bel{#1}}}
\def\ref@#1{\@ref@{#1}\putr@fl@bel{#1}}
\def\pageref@#1{\@pageref@{#1}\putr@fl@bel{#1}}
\def\fullref@#1{\@fullref@{#1}\putr@fl@bel{#1}}
\def\Fullref@#1{\@Fullref@{#1}\putr@fl@bel{#1}}
\def\reftype@#1{\@reftype@{#1}\putr@fl@bel{#1}}
\def\Reftype@#1{\@Reftype@{#1}\putr@fl@bel{#1}}

\def\ref@@#1{\leavevmode\refc@lor{\rm$\langle$#1$\rangle$}}
\let\pageref@@\ref@@
\let\Fullref@@\ref@@
\let\fullref@@\ref@@
\let\reftype@@\ref@@
\let\Reftype@@\ref@@

\def\bibitem{\@ifnextchar[{\bibitem@@}{\bibitem@@@}}
\def\bibitem@@[#1]#2{\@bibitem@[#1]{#2}\putl@bel{#2}}
\def\bibitem@@@#1{\@bibitem@{#1}\putl@bel{#1}\ignorespaces}

\def\cit@{\@ifnextchar[{\@cit@@@}{\@cit@@}}
\def\@cit@@#1{\@cite@{#1}{\let\labc@lor\citc@lor\putl@bel{#1}}}
\def\@cit@@@[#1]#2{\@cite@[#1]{#2}{\let\labc@lor\citc@lor\putl@bel{#2}}}
\def\cit@@{\@ifnextchar[{\cit@@@@}{\cit@@@}}
\def\cit@@@#1{\leavevmode{\citc@lor\rm[#1]}}
\def\cit@@@@[#1]#2{\leavevmode{\citc@lor\rm[#2, #1]}}

\def\showcitationstrue{\let\cite\cit@}
\def\showcitationsfalse{\let\cite\cit@@}

\def\showreferencestrue{%
  \let\ref\ref@\let\pageref\pageref@%
  \let\fullref\fullref@\let\Fullref\Fullref@%
  \let\reftype\reftype@\let\Reftype\Reftype@}
\def\showreferencesfalse{%
  \let\ref\ref@@\let\pageref\pageref@@%
  \let\fullref\fullref@@\let\Fullref\Fullref@@%
  \let\reftype\reftype@@\let\Reftype\Reftype@@}

\def\showlabelstrue{\let\putl@bel\putl@b@l}
\def\showlabelsfalse{\let\putl@bel\hid@@}

\def\postit@{\@ifnextchar[{\postit@@}{\p@tp@stit}}
\def\postit@@[#1]{\postit@@@#1,@}
\def\postit@@@#1,{\@ifnextchar{@}{\p@@tp@stit{#1}}{\postit@@@@#1,}}
\def\postit@@@@#1,#2,@{\p@@@tp@stit{#1}{#2}}

\long\def\p@tp@stit#1{\put[0pt,.25\baselineskip]{%
  \hbox{\postitc@lor{\fontfamily{cmr}\fontshape{n}\fontseries{m}\selectfont%
  \tiny\setbox5\hbox{\vphantom{X}\smash{\ns#1}}%
  \hbox to 0pt{\hss\tiny$\blacktriangledown$\kern-.085em}%
  \raise2.25ex\hbox to 0pt{\hss\framebox{\box5}}}}}}

\long\def\p@@tp@stit#1@#2{\put[0pt,.25\baselineskip]{%
  \hbox{\postitc@lor{\fontfamily{cmr}\fontshape{n}\fontseries{m}\selectfont%
  \tiny\setbox5\hbox{\vbox{\hsize#1\leftskip\z@\raggedright
  \parindent\z@{\ns#2\par}\vss}}%
  \hbox to 0pt{\hss\tiny$\blacktriangledown$\kern-.085em}%
  \raise2.25ex\hbox to 0pt{\hss\framebox{\box5}}}}}}

\long\def\p@@@tp@stit#1#2#3{\put[0pt,.25\baselineskip]{%
  \hbox{\postitc@lor{\fontfamily{cmr}\fontshape{n}\fontseries{m}\selectfont%
  \tiny\setbox5\hbox{\vbox to #2{\hsize#1\leftskip\z@\raggedright
  \parindent\z@{\ns#3\par}\vss}}%
  \hbox to 0pt{\hss\tiny$\blacktriangledown$\kern-.085em}%
  \raise2.25ex\hbox to 0pt{\hss\framebox{\box5}}}}}}

\def\postitc@lor{\color{postitcolor}}

\def\showpostittrue{\let\postit\postit@}
\def\showpostitfalse{\let\postit\hid@@@}

\long\def\hid@@#1{\ignorespaces}
\def\hid@@@{\@ifnextchar[{\hid@@@@}{\hid@@}}
\long\def\hid@@@@[#1]{\hid@@}

\makeatother

\papersize{216truemm}{279truemm}
\margins{28truemm}{21truemm}
\advance\voffset-5mm
\advance\hoffset2.5truemm
\advance\footskip2.5truemm

\headline{\hfill}
\footline{\small\hfill--\kern1ex\thepage\kern1ex--\hfill}

\sectionbeforeskip{1.5\bigskipamount}
\sectionstyle{\centering\normalsize\sc}
\sectionafterskip{\bigskipamount}
\sectionafterindenttrue

\subsectionbeforeskip{\bigskipamount}
\subsectionstyle{\normalsize\bf}
\subsectionafterskip{.5\bigskipamount}
\subsectionafterindenttrue

\paragraphbeforeskip{\medskipamount}
\paragraphstyle{\ft{cmbxsl10}\boldmath}
\paragraphafterskip{1.25ex}
\paragraphindenttrue
\paragraphdot{.}

\abovedisplayskip\smallskipamount
\belowdisplayskip\smallskipamount
\abovedisplayshortskip\smallskipamount
\belowdisplayshortskip\smallskipamount


\textfloatsep\floatsep

\proofingfalse
\autolabelfalse

\showlabelsfalse
\showcitationstrue
\showreferencestrue

\showpostittrue

\showfigurestrue
\showdiagramstrue

\newtheorem{stat}{\statname}  \unnumbered{stat}

\newtheorem{nstat}{\nstatname}[section]

\newtheorem[{\ns}{}]{definition}[nstat]{Definition}
\newtheorem{lemma}[nstat]{Lemma}
\newtheorem{proposition}[nstat]{Proposition}
\newtheorem{theorem}[nstat]{Theorem}

\newtheorem[{\ns}{}]{exercise}[nstat]{Exercise}
\newtheorem[{\ns}{}]{example}[nstat]{Example}
\newtheorem[{\ns}{}]{remark}[nstat]{Remark}

\let\ns\normalshape

\captionstyle{}
\captionintrostyle{}

\def~{\kern0.33em}

\def\brasl#1{$($\hbox to .5em{\hss\textsl{#1\kern.05em}\hss}$)$}

\let\0\emptyset

\def\fieldstyle#1{{\mathbb #1}}
\def\calstyle#1{{\mathcal #1}}

\newcommand{\R}{\fieldstyle R}

\newcommand{\Z}{\fieldstyle Z}
\newcommand{\Q}{\fieldstyle Q}

\newcommand{\T}{{\calstyle T}}

\newcommand{\E}{{\mathrm{E}}}
\newcommand{\SE}{{\mathrm{SE}}}
\renewcommand{\O}{{\mathrm{O}}}
\newcommand{\SO}{{\mathrm{SO}}}

\renewcommand{\(}{(\mkern-3.5mu(}
\renewcommand{\)}{)\mkern-3.5mu)}

\begin{document}

\title{\large\bf ON THE GENERIC TRIANGLE GROUP}

\author{
\normalsize\textsl{Stefano Isola}
and
\normalsize\textsl{Riccardo Piergallini}\\[2pt]
\normalsize Scuola di Scienze e Tecnologie\\[-3pt]
\normalsize\textsl{Universit\`a di Camerino -- Italy}
}

\date{}

\maketitle

\vglue-3mm\vglue0pt

\begin{abstract}\vglue3mm\noindent
We introduce the concept of a generic Euclidean triangle $\tau$ and study the group $G_\tau$ generated by the reflection across the edges of $\tau$. In particular, we prove that the subgroup $T_\tau$ of all translations in $G_\tau$ is free abelian of infinite rank, while the index 2 subgroup $H_\tau$ of all orientation preserving transformations in $G_\tau$ is free metabelian of rank 2, with $T_\tau$ as the commutator subgroup. As a consequence, the group $G_\tau$ cannot be finitely presented and we provide explicit minimal infinite presentations of both $H_\tau$ and $G_\tau$. This answers in the affirmative the problem of the existence of a minimal presentation for the free metabelian group of rank 2. Moreover, we discuss some examples of non-trivial relations in $T_\tau$ holding for given non-generic triangles $\tau$.

\end{abstract}


\section{Introduction\label{intro/sec}}

The term \textsl{triangle group} is generally reserved in the literature to the group $G_\tau$ generated by the reflections $r_1,r_2,r_3$ across the sides of an Euclidean, spherical or hyperbolic triangle $\tau$ with internal angles $\alpha_i = \pi/n_i$, where the specific geometry depends on $1/n_1 + 1/n_2 + 1/n_3$ being $=1$, $> 1$ or $< 1$, respectively. The structure of these groups is well understood since the seminal works by Fricke and Klein \cite{FrKl97} and Coxeter \cite{Co34} (see also \cite{CoMo79} and \cite{Mil75}). In particular, based on the fact that the triangle $\tau$ tiles the plane or the sphere, we have the finite presentation $G_\tau = \langle x_1,x_2,x_3 \,|\, x_1^2,x_2^2,x_3^2,(x_2x_3)^{n_1},(x_3x_1)^{n_2},(x_1x_2)^{n_3}\rangle$ with the symbol $x_i$ corresponding to the reflection $r_i$.

The hyperbolic case (the only non-trivial one) has been widely studied, with a special focus on its complex version, and several notions of generalized triangle groups have been considered, in terms of finite presentation independently on the original geometric setting. 

But little seems to be know about the group $G_\tau$ for more general triangles $\tau$, even if the strong hypothesis on the angles $\alpha_i$ is just mildly relaxed to include rational multiples of $\pi$, that is in the case of $\tau$ a \textsl{rational triangle} with $\alpha_i = m_i \pi/n_i$.

In this note we study the group $G_\tau$ for an arbitrary Euclidean triangle $\tau$, starting from the case of \textsl{generic triangles}. These are introduced in Section \ref{triangle/sec} as the triangles whose edge lengths are algebraically independent over the rationals (up to a common factor), and can be considered as the opposite to the rational triangles in the spectrum of all Euclidean triangles. In particular, we show that generic triangles include \textsl{typical triangles}, the ones whose angles are linearly independent over the rational.

Our main results on the structure of the \textsl{triangle group} $G_\tau$ for a generic triangle $\tau$, are presented in Sections \ref{translations/sec} and \ref{rotations/sec}, after a brief discussion of some generalities about $G_\tau$ and its linearization $S_\tau \subset \O(2)$ in Section \ref{reflections/sec}. They concern the \textsl{translation subgroup} $T_\tau$, consisting of all translations in $G_\tau$, and the \textsl{rotation subgroup} $H_\tau$, the index 2 subgroup of all orientation preserving transformations in $G_\tau$. Namely, we show that $T_\tau$ is free abelian of infinite rank generated by the translation $t_1 = (r_1r_2r_3)^2$ and its conjugates in $G_\tau$ (Theorem \ref{T-tau/thm}), while $H_\tau$ is free metabelian of rank 2 generated by the rotations $r_2r_1$ and $r_1r_3$ (Theorem \ref{H-tau/thm}), with $[H_\tau,H_\tau] = T_\tau$. As a consequence, $G_\tau$ cannot be finitely presented. Moreover, we provide explicit presentations for $H_\tau$ (Theorem \ref{H-tau-min/thm}) and $G_\tau$ (Theorem \ref{G-tau-min/thm}), which are minimal in the sense that no relation can be removed without changing the group (Theorem \ref{minimality/thm}). From the purely group theoretical viewpoint, this solves the problem of finding a minimal presentation for the free metabelian group of rank 2 (cf. \cite{BCGS14}).

Finally, in Section \ref{examples/sec} we discuss some examples of non-trivial relations in $T_\tau$, holding for continuous families of non-generic but typical triangles $\tau$ and for certain isolated such triangles, respectively.


\section{Generic triangles\label{triangle/sec}}

Given an Euclidean triangle $\tau = A_1A_2A_3$, let $\ell_i > 0$ denote the length of the edge $e_i=A_jA_k$ and $\alpha_i > 0$ denote the measure in radians of the (non-oriented) interior angle $A_jA_iA_k$, with $\{i,j,k\} = \{1,2,3\}$.

\begin{definition}\label{generic/def}
We call $\tau$ a \textsl{generic triangle} if for some $k > 0$ (hence for almost every $k \in \R$) the real numbers $k \ell_1, k \ell_2$ and $k \ell_3$ are \textsl{algebraically independent} (over the rationals), namely it does not exist any non-trivial polynomial $p(x_1,x_2,x_3) \in \Z[x_1,x_2,x_3]$ such that $p(k \ell_1, k \ell_2, k \ell_3) = 0$.
\end{definition}

A different formulation of the above condition is that for every non-trivial polynomial $p(x_1,x_2,x_3) \in \Z[x_1,x_2,x_3]$ the polynomial $q(x) = p(\ell_1 x,\ell_2 x,\ell_3 x)$ is non-trivial in $\R[x]$, or equivalently the field $\Q(\ell_1 x,\ell_2 x,\ell_3 x)$ has transcendence degree 3 over $\Q$ (see \cite[Sec. 6.4]{Rot03}). Then a straightforward argument on cardinalities shows that generic triangles are dense in the space of all Euclidean triangles (the same argument also explains why ``some $k$'' could be replaced by ``almost every $k$'' in the definition).

Next proposition just translates Definition \ref{generic/def} in terms of trigonometric functions of the interior angles of the triangle $\tau$.

\begin{proposition}\label{trig/thm}
An Euclidean triangle $\tau$ as above is a generic triangle if and only one of the following equivalent properties holds:
\begin{enumerate}
\item[\brasl{s}] for some (hence almost every) $k \in \R$ the real numbers $k \sin \alpha_1, k \sin \alpha_2$ and $k \sin \alpha_3$ are algebraically independent over the rationals;
\item[\brasl{c}] for some (hence almost every) $k \in \R$ the real numbers $k \cos \alpha_1, k \cos \alpha_2$ and $k \cos \alpha_3$ are algebraically independent over the rationals.
\end{enumerate}
\end{proposition}

\begin{proof}
The equivalence of the condition in Definition \ref{generic/def} with property \brasl{s} immediately follows by the law of sines, while some work is needed to verify the equivalence between \brasl{s} and \brasl{c}. Once these properties are reformulated in terms of extensions of $\Q$ involving the indeterminate $x$ as above, we are reduced to proving that $\Q(x \sin \alpha_1, x \sin \alpha_2,x \sin \alpha_3)$ and $\Q(x \cos \alpha_1, x \cos \alpha_2,x \cos \alpha_3)$ have the same transcendence degree over $\Q$. By elementary trigonometry, from $\alpha_1 + \alpha_2 + \alpha_3 = \pi$ we get the equations
\begin{equation}\label{cos/eqn}
\cos^2\alpha_1 + \cos^2\alpha_2 + \cos^2\alpha_3 + 2 \cos\alpha_1 \cos\alpha_2 \cos\alpha_3 = 1\,,\end{equation}
\vspace{-18pt}
\begin{equation}\label{sin/eqn}
\begin{array}{rcl}\vrule width0pt depth6pt height18pt
\sin^4\alpha_1 + \sin^4\alpha_2 + \sin^4\alpha_3 + 4 \sin^2\alpha_1 \sin^2\alpha_2 \sin^2\alpha_3 + {} \kern20pt&&\cr {} - 2 \sin^2\alpha_1 \sin^2\alpha_2 - 2 \sin^2\alpha_1 \sin^2\alpha_3 - 2 \sin^2\alpha_2 \sin^2\alpha_3 &=& 0\,.
\vrule width0pt depth12pt height6pt\end{array}
\end{equation}
Multiplying (\ref{cos/eqn}) by $x^3$, we see that $x$ is algebraic over $\Q(x \cos \alpha_1, x \cos \alpha_2,x \cos \alpha_3)$, hence the possibly larger extension $\Q(x \cos \alpha_1, x \cos \alpha_2,x \cos \alpha_3,x)$ has the same transcendence degree as $\Q(x \cos \alpha_1, x \cos \alpha_2,x \cos \alpha_3)$ over $\Q$. Similarly, multiplying (\ref{sin/eqn}) by $x^6$, we see that $x$ is algebraic over $\Q(x \sin \alpha_1, x \sin \alpha_2,x \sin \alpha_3)$ as well, hence the transcendence degree of $\Q(x \sin \alpha_1, x \sin \alpha_2,x \sin \alpha_3,x)$ over $\Q$ is the same as that of $\Q(x \sin \alpha_1, x \sin \alpha_2,x \sin \alpha_3)$. Now, the relations $x^2\sin^2\alpha_i + x^2\cos^2\alpha_i = x^2$ allow us to conclude that $\Q(x \cos \alpha_1, x \cos \alpha_2,x \cos \alpha_3,x)$ and $\Q(x \sin \alpha_1, x \sin \alpha_2,x \sin \alpha_3,x)$ have the same transcendence degree over $\Q$.
\end{proof}

As noticed in the above proof, due to the relation $\alpha_1 + \alpha_2 + \alpha_3 = \pi$ we cannot have $k = 1$ (or any rational number) in points \brasl{s} and \brasl{c} of Proposition \ref{trig/thm}. However, if $\tau$ is a generic triangle then $\ell_1, \ell_2, \ell_3$, as well as $\sin \alpha_1, \sin \alpha_2, \sin \alpha_3$ and $\cos \alpha_1, \cos \alpha_2, \cos \alpha_3$, form linearly independent triples over the rationals, being linear independence a homogeneous condition where the factor $r > 0$ can be canceled.

\begin{lemma}\label{generic/thm}
If an algebraic relation $p(\sin\alpha_1, \sin\alpha_2, \sin\alpha_3, \cos\alpha_1, \cos\alpha_2, \cos\alpha_3) = 0$, with $p(x_1,x_2,x_3,x_4,x_5,x_6) \in \Z[x_1,x_2,x_3,x_4,x_5,x_6]$, holds for a generic triangle, then it holds for every triangle.
\end{lemma}

\begin{proof}
By performing in
\begin{equation}\label{p/eqn}
p(\sin\alpha_1, \sin\alpha_2, \sin\alpha_3, \cos\alpha_1, \cos\alpha_2, \cos\alpha_3) = 0
\end{equation}
the replacements 
\begin{equation}\nonumber
\sin\alpha_i = \frac{\sqrt{\vrule width0pt height10pt depth3pt
\smash{(\ell_1^2 + \ell_2^2 + \ell_3^2)^2 - 2(\ell_1^4 + \ell_2^4 + \ell_3^4)}}}{2\ell_j \ell_k}
\text{ \ and \ }
\cos\alpha_i = \frac{\ell_i^2 - \ell_j^2 - \ell_k^2}{2 \ell_j \ell_k}
\end{equation}
given by the laws of sines and cosines, we obtain
\begin{equation}\label{q/eqn}
q(\ell_1, \ell_2, \ell_3, \!\textstyle\sqrt{\vrule width0pt height10pt depth3pt \smash{(\ell_1^2 + \ell_2^2 + \ell_3^2)^2 - 2(\ell_1^4 + \ell_2^4 + \ell_3^4)}}\,) = 0\,,
\end{equation}
with $q(x_1,x_2,x_3,x_4) \in \Z[x_1,x_2,x_3,x_4]$. 

Considering $q(x_1,x_2,x_3,x_4)$ as a polynomial in the indeterminate $x_4$ with coefficients in $\Z[x_1,x_2,x_3]$ and dividing it by $x_4^2 - (x_1^2 + x_2^2 + x_3^2)^2 + 2(x_1^4 + x_2^4 + x_3^4)$, we get a binomial $a(x_1,x_2,x_3)\kern1pt x_4 + b(x_1,x_2,x_3)$ as the remainder, with $a(x_1,x_2,x_3)$ and $b(x_1,x_2,x_3)$ in $\Z[x_1,x_2,x_3]$.

Then, equation (\ref{q/eqn}) turns out to be equivalent to
\begin{equation}\nonumber
a(\ell_1, \ell_2, \ell_3)
\textstyle\sqrt{\vrule width0pt height10pt depth3pt \smash{(\ell_1^2 + \ell_2^2 + \ell_3^2)^2 - 2(\ell_1^4 + \ell_2^4 + \ell_3^4)}} = -\,b(\ell_1, \ell_2, \ell_3)\,,
\end{equation}
which squared gives the integral algebraic relation 
\begin{equation}\label{r/eqn}
a(\ell_1, \ell_2, \ell_3)^2\,[(\ell_1^2 + \ell_2^2 + \ell_3^2)^2 - 2(\ell_1^4 + \ell_2^4 + \ell_3^4)] = b(\ell_1, \ell_2, \ell_3)^2\,.
\end{equation}
Taking into account that the replacing expressions in (\ref{p/eqn}) are homogeneous of degree 0 on $\ell_1, \ell_2$ and $\ell_3$, we have that also the equation (\ref{r/eqn}) must be homogeneous.

Now, if the relation (\ref{p/eqn}) holds for the angles $\alpha_1, \alpha_2$ and $\alpha_3$ of a generic triangle, then relation (\ref{r/eqn}) holds for the lengths $\ell_1, \ell_2$ and $\ell_3$ of its edges, even if these are scaled by any factor $k > 0$ (by the homogeneity). By definition of a generic triangle, this implies that the polynomial
\begin{equation}\nonumber
a(x_1, x_2, x_3)^2\,[(x_1^2 + x_2^2 + x_3^2)^2 - 2(x_1^4 + x_2^4 + x_3^4)] - b(x_1, x_2, x_3)^2
\end{equation}
is trivial in $\Z[x_1,x_2,x_3]$. Then, equations (\ref{r/eqn}) and (\ref{q/eqn}) are 
identically satisfied and we can conclude that (\ref{p/eqn}) holds for every triangle.
\end{proof}

A deeper analysis of the algebraic dependence of $\cos\alpha_1, \cos\alpha_2$ and $\cos\alpha_3$, shows that generic triangles are typical, in the sense of the following definition (see \cite{GaStVo92}).

\begin{definition}\label{typical/def}
An Euclidean triangle $\tau$ as above is called a \textsl{typical triangle} if the real numbers $\alpha_1,\alpha_2$ and $\alpha_3$ are linearly independent over the rationals.
\end{definition}

\begin{proposition}\label{typical/thm}
Generic Euclidean triangles are typical.
\end{proposition}

\begin{proof}
Let $\tau$ a generic triangle. We want to prove that equation (\ref{cos/eqn}) is essentially the only algebraic relation between the cosines of the interior angles of $\tau$, being any polynomial $p(x_1,x_2,x_3) \in \Z[x_1,x_2,x_3]$ such that $p(\cos\alpha_1,\cos\alpha_2,\cos\alpha_3) = 0$ divisible by
\begin{equation}\label{basic/eqn}
\vrule width0pt depth12pt
x_1^2 + x_2^2 + x_3^2 + 2x_1x_2x_3 - 1\,.
\end{equation}

Assume that the identity $p(\cos\alpha_1, \cos\alpha_2, \cos\alpha_3) = 0$ holds for $\tau$.
Then, according to Lemma \ref{generic/thm}, it must hold for any triangle, and by using once again the relation $\alpha_1 + \alpha_2 + \alpha_3 = \pi$ we have $p(-\cos(\alpha_2 + \alpha_3),\cos\alpha_2, \cos\alpha_3) = 0$ for every $\alpha_2,\alpha_3 \geq 0$\break such that $\alpha_2 + \alpha_3 < \pi$. Therefore, $p(x_1,x_2,x_3)$ is divisible by both the linear binomials\break $x_1 + x_2x_3 \pm \sqrt{(1-x_2^2)(1-x_3^2)}$ in the indeterminate $x_1$ with coefficients in the quadratic closure of the field of fractions $\Q(x_2,x_3)$, hence it is divisible by $x_1^2 + x_2^2 + x_3^2 + 2x_1x_2x_3 - 1$ in $\Z[x_1,x_2,x_3]$ (notice that the last polynomial is monic with respect to $x_1$).

Now, by contradiction, let $n_1 \alpha_1 + n_2 \alpha_2 + n_3 \alpha_3 = 0$ be a non-trivial vanish\-ing linear combination of the interior angles $\alpha_1, \alpha_2$ and $\alpha_3$ of a generic triangle, with integral coefficients $n_1,n_2$ and $n_3$, which
can be assumed coprime without loss of generality. By elementary trigonometry, we have
\begin{equation}\nonumber
\cos (n_1 \alpha_1) - \cos (n_2 \alpha_2) \cos (n_3 \alpha_3) = \sin (n_2 \alpha_2)\sin (n_3 \alpha_3)\,.\end{equation} 
Squaring and using the Pythagorean identity, we readily obtain 
\begin{equation}\nonumber
T^2_{n_1\!}(\cos\alpha_1) + T^2_{n_2\!}(\cos\alpha_2) + T^2_{n_3\!}(\cos\alpha_3) - 2\,T_{n_1\!}(\cos\alpha_1) T_{n_2\!}(\cos\alpha_2) T_{n_3\!}(\cos\alpha_3) = 1\,,
\end{equation}\nonumber
where $T_n(x) \in \Z[x]$ denotes the Chebyshev polynomials defined by the identity $T_n(\cos \alpha) = \cos (n\alpha)$.
By the above, the polynomial 
\begin{equation}\label{cheb/eqn}
T^2_{n_1\!}(x_1) + T^2_{n_2\!}(x_2) + T^2_{n_3\!}(x_3) - 2\,T_{n_1\!}(x_1)T_{n_2\!}(x_2)T_{n_3\!}(x_3) - 1
\end{equation}
must be divisible by (\ref{basic/eqn}). Since $T_n(1) = 1$ for all $n$, setting $x_2 = x_3 = 1$ both in (\ref{basic/eqn}) and (\ref{cheb/eqn}),  we get $(x_1+1)^2=0$ and $(T_{n_1}(x_1)-1)^2=0$, respectively. Thus, $T_{n_1}(x_1) - 1$ must be divisible by $x_1 + 1$ in $\Z[x_1]$. This implies that $n_1$ is even, because $T_n(x)$ has same parity of $n$. By the symmetry of (\ref{basic/eqn}) and (\ref{cheb/eqn}), the same argument shows that $n_2$ and $n_3$ must be even as well, contradicting the coprimality assumption.
\end{proof}

In the light of Proposition \ref{typical/thm}, generic triangles can be somewhat thought of as the opposite end in the spectrum of all Euclidean triangles with respect to the rational ones.



\section{The triangle group\label{reflections/sec}}

For any Euclidean triangle $\tau$ we denote by $G_\tau = \langle r_1, r_2, r_3 \rangle \subset \E(2)$ the subgroup of the group $\E(2)$ of the Euclidean isometries of the plane generated by the reflections $r_1, r_2$ and $r_3$ across the edges $e_1, e_2$ and $e_3$ of $\tau$, respectively.
We call $G_\tau$ the \textsl{triangle group} of $\tau$.

The standard exact sequence 
\begin{equation}\nonumber
1 \longrightarrow \R^2 \buildrel \iota \over \longrightarrow \E(2) \buildrel \lambda \over \longrightarrow \O(2) \longrightarrow 1\,,
\end{equation}
where $\iota$ is the inclusion of $\R^2$ in $\E(2)$ as the subgroup of translations and $\lambda$ is the linearization homomorphism, induces by restriction the exact sequence 
\begin{equation}\label{Gt-exact/eqn} 
1 \longrightarrow T_\tau \buildrel \iota_\tau \over \longrightarrow G_\tau \buildrel \lambda_\tau \over \longrightarrow S_\tau \longrightarrow 1\,,
\end{equation} 
where $T_\tau \subset G_\tau$ is the \textsl{translation subgroup} consisting of all translations in $G_\tau$, while $S_\tau = \lambda(G_\tau) = \langle s_1, s_2, s_3 \rangle \subset \O(2)$ with $s_i = \lambda(r_i)$ the linearization of $r_i$.

We observe that the structure of the group $G_\tau$ (including the latter exact sequence) is invariant under similarities. Hence, without loss of generality we can assume that the incircle of the triangle $\tau$ is coincides with the unit circle centered at the origin. Under this assumption, we have 
\begin{equation}\label{ri/eqn}
(x)r_i = (x)s_i + 2v_i
\end{equation} 
for every $x \in \R^2$ (here and in the following we use the right notation for the action of $G_\tau$), where $v_i$ is the unit vector from the origin to the tangency point of the edge $e_i$ and the incircle of $\tau$, as shown in Figure \ref{figure1/fig}.

\begin{Figure}[htb]{figure1/fig}
\fig{}{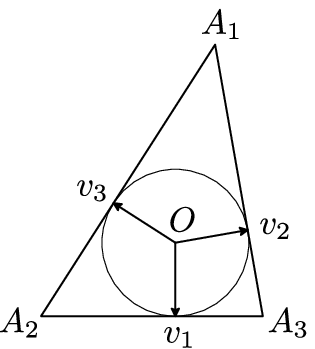}[-3pt]
    {}{The unit vectors $v_1,v_2$ and $v_3$.}
\end{Figure}

We want to determine a minimal presentation of the group $S_\tau$ in the case when $\tau$ is a typical, and hence $S_\tau$ is a dense subgroup of $\O(2)$. In order to do that, we first recall from \cite{GaStVo92} the definition of a stable sequence and the stability criterion for a product of generators of $S_\tau$ to be trivial.

\begin{definition}\label{stable/def}
A sequence $i_1i_2 \dots i_n$ of symbols from $\{1,2, \dots, N\}$ is called a \textsl{stable sequence} if its terms can be paired into disjoint pairs of identical symbols, one located at an odd and the other at an even position. Differently said, the length $n$ of the sequence is even and the symbolic alternating sum $i_1 - i_2 + \dots + i_{n-1} - i_n$ vanish (as an algebraic sum of symbols, not of integers).
\end{definition}

The motivation for the term ``stable'' is that a stable sequence as above represents the sequence of sides of a polygonal billiard (whose $N$ sides are arbitrarily numbered) visited by a $n$-periodic trajectory, which is stable in the sense that it survives to small perturbations of the polygon (see \cite{GaStVo92}). Actually, also Lemma \ref{stable/thm} below and its proof are essentially translated from the context of stable trajectories in polygonal billiards, focusing on the case $N = 3$.

Before going on, let us give an operational characterization of stability.

\begin{lemma}\label{stable-red/thm}
A sequence $i_1i_2 \dots i_n$ is stable if and only if one can reduce it to the empty sequence by a finite number of operation of the following types: 
\begin{enumerate}
\item[\brasl{a}] transposition of two adjacent subsequences both consisting of two symbols;
\item[\brasl{b}] deletion of a subsequence consisting of two identical symbols.
\end{enumerate}
\end{lemma}

\begin{proof}
First of all, we note that both operations and the inverse of the second one, that is the insertion of two adjacent identical symbols in a sequence, all preserve the parity of the position of each term in the sequence, hence they preserve stability. This immediately gives the ``if'' part of the statement, being the empty sequence stable.

The ``only if'' part can be proved by induction on the length of the sequence, starting  once again from the empty sequence. For the inductive step, assume we are given any non-empty stable sequence $i_1i_2 \dots i_n$. The stability implies that $i_{2k-1} = i_2$ for some\break $1 \leq k \leq n/2$. If $k = 1$, we can reduce the length of the sequence by deleting the subsequence $i_1i_2$. Otherwise, by $k - 2$ transpositions of pairs, we get a sequence starting with the four symbols $i_1i_2i_{2k+1}i_{2k+2}$, then we can reduce the length of the word by deleting the subsequence $i_2i_{2k+1}$.
\end{proof}

\begin{lemma}\label{stable/thm}
If a sequence $i_1i_2 \dots i_n$ of symbols from $\{1,2,3\}$ is stable, then the product $s_{i_1}s_{i_2} \dots s_{i_n}$ is the identity in $S_\tau$. Moreover, for a typical triangle $\tau$ the stability of the sequence $i_1i_2 \dots i_n$ is also necessary in order  $s_{i_1}s_{i_2} \dots s_{i_n}$ to be the identity. 
\end{lemma}

\begin{proof} We proceed in the same spirit as in \cite[Sec. 6.B]{GaStVo92}. 
We first orient the egdes $e_1,e_2$ and $e_3$ in the counterclockwise way along the boundary of the triangle $\tau$, and denote by $\beta_i$ the oriented angle from $e_1$ (fixed as reference vector) to $e_i$. Then, we have $\beta_1 = 0$, $\beta_2 = \pi - \alpha_3$ and $\beta_3 = \pi + \alpha_2$ mod $2\pi$. Moreover, any composition $s_js_k$ gives the linear rotation of angle $2(\beta_k - \beta_j)$ mod $2\pi$, and hence any product $s_{i_1}s_{i_2} \dots s_{i_n}$ with $n$ even gives the linear rotation of angle 
\begin{equation}\label{phi/eqn}
\phi = 2(\beta_{i_2}-\beta_{i_1})+ \cdots + 2(\beta_{i_n}-\beta_{i_{n-1}}) \text{ mod }2\pi\,.
\end{equation}

Now, the stability of the sequence $i_1i_2\dots i_n$ implies that $\phi = 0$ mod $2\pi$ and thus $s_{i_1}s_{i_2} \dots s_{i_n}$ is the identity in $S_\tau$.

In the opposite direction, start with a product $s_{i_1}s_{i_2} \dots s_{i_n}$ that gives the identity. Then, $n$ must be even and (\ref{phi/eqn}) can be rewritten in terms of the $\alpha_i$'s by using the above identities. If the sequence $i_1i_2\dots i_n$ is not stable, this yields a non-trivial rational linear relation among the angles $\alpha_2, \alpha_3$ and $\pi$, which implies that $\tau$ is not typical. 
\end{proof}

At this point, we are in position to obtain the wanted presentation of $S_\tau$.

\begin{proposition}\label{S-tau/thm}
For a typical triangle $\tau$ the group $S_\tau$ admits the finite presentation $$\langle x_1,x_2,x_3 \,|\, x_1^2, x_2^2, x_3^2, (x_1x_2x_3)^2\,\rangle\,,$$ with the symbols $x_1,x_2,x_3$ corresponding to $s_1,s_2,s_3$, respectively.\end{proposition}

\begin{proof}
According to Lemma \ref{stable/thm}, all the four relations of the presentation hold in $S_\tau$, being the corresponding sequences of indices stable.

Viceversa, Lemmas \ref{stable-red/thm} and \ref{stable/thm} say that any word $x_{i_1}x_{i_2}\dots x_{i_n}$ representing the identity in $S_\tau$ can be reduced to the empty word by canceling squared terms $x_i^2$ and commuting products $x_ix_j$ and $x_kx_l$. So, to conclude the proof it is enough to show that any commutator $[x_ix_j,x_kx_l]$ is the identity modulo the given four relations. Up to inversions, the only non-trivial cases are $[x_1x_2,x_1x_3]\,, [x_1x_2,x_2x_3]$ and $[x_1x_3,x_2x_3]$. For these we have: $[x_1x_2,x_1x_3] = x_1x_3(x_1x_2x_3)^{-2}x_3x_1$ and $[x_1x_2,x_2x_3] = [x_1x_3,x_2x_3] = x_1(x_1x_2x_3)^{-2}x_1$.\end{proof}



\section{The translation subgroup\label{translations/sec}}

Due to Lemma \ref{stable/thm} and the exact sequence (\ref{Gt-exact/eqn}), for any Euclidean triangle $\tau$ the product $r_{i_1}r_{i_2} \dots r_{i_n}$ gives a translation in $T_\tau$ if the sequence $i_1i_2 \dots i_n$ is stable. On the other hand, when the triangle $\tau$ is typical we obtain in this way all the translations in $T_\tau$, and it is clear from Proposition \ref{S-tau/thm} that a special role is played by the minimal stable product $t_1 = (r_1r_2r_3)^2$, coming from the code-word of the Fagnano trajectory, the simplest stable periodic trajectory in any acute triangle (see \cite{GaStVo92}). 

We denote the conjugation class of $t_1$ in $G_\tau$ by
\begin{equation}\nonumber
C(t_1) = \{\(t_1\)g = g^{-1}t_1g \,|\, g \in G_\tau\} \subset T_\tau\,.
\end{equation}

\begin{Figure}[b]{figure2/fig}
\vskip12pt
\fig{}{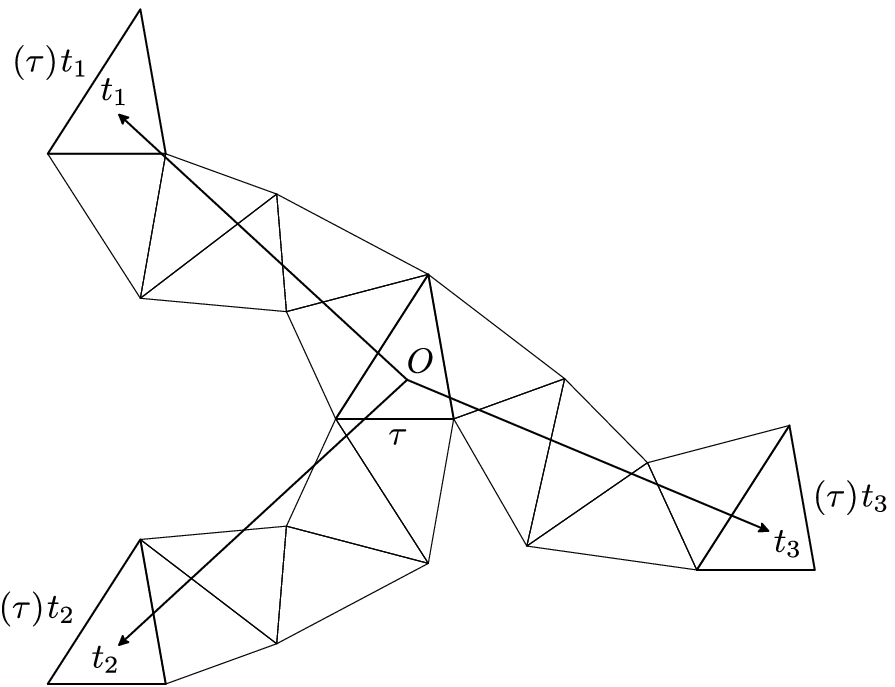}
    {}{The translations $t_1,t_2$ and $t_3$}
\end{Figure}

By identifying translations in $T_\tau$ with the corresponding vectors in $\R^2$ and considering the natural action of $S_\tau \subset \O(2)$ on them, for $g = r_{i_1}r_{i_2}\dots r_{i_n} \in G_\tau$ the conjugate $\(t_1\)g$ is given by
\begin{equation}\label{conjugate/eqn}
\(t_1\)r_{i_1}r_{i_2}\dots r_{i_n} = \(t_1\)s_{i_1}s_{i_2}\dots s_{i_n}\,,
\end{equation}
according to equation (\ref{ri/eqn}).
In the case when $\tau$ is of a typical triangle, the density of the subgroup $S_\tau \subset \O(2)$ implies that $C(t_1)$ forms a dense subset of the circle $\rho S^1 \subset R^2$ of radius 
\begin{equation}\nonumber
\rho = \|t_1\| = 4(\sin \alpha_1 + \sin \alpha_2 + \sin \alpha_3)\,.
\end{equation}

The translation $t_1$ and its conjugates $t_2 = (r_2r_3r_1)^2 = \(t_1\)r_1$ and $t_3 = (r_3r_1r_2)^2 = \(t_1\)r_3$, are represented in Figure \ref{figure2/fig} (here we assume the same numbering as in the previous Figure \ref{figure1/fig} for the edges of the triangle $\tau$).

\begin{theorem}\label{T-tau/thm}
For a typical triangle $\tau$ the translation subgroup $T_\tau \subset G_\tau$ is normally generated by the translation $t_1 = (r_1r_2r_3)^2$. Moreover, if $\tau$ is generic then $T_\tau$ is a free abelian group having as a basis the conjugation class $C(t_1)$.
\end{theorem}

\begin{proof}
Given any $t \in T_\tau$ with $\tau$ a typical triangle, we can express it as a product $r_{i_1} r_{i_2} \dots r_{i_n}$ of generators of $G_\tau$. In view of the exact sequence (\ref{Gt-exact/eqn}), the corresponding product $s_{i_1} s_{i_2} \dots s_{i_n}$ gives the identity in $S_\tau$, hence the sequence $i_1 i_2 \dots i_n$ is stable by Lemma \ref{stable/thm}. Then, arguing as in the proof of Proposition \ref{S-tau/thm}, we can rewrite $r_{i_1} r_{i_2} \dots r_{i_n}$ as a product of conjugates of $r_1^2, r_2^2,r_3^2$ and $(r_1r_2r_3)^2$. Since the $r_i^2$'s are trivial in $G_\tau$, we can conclude that $t$ is a product of conjugates of $t_1$, which gives the first part of the theorem.

Now, assume that $\tau$ is a generic triangle. We have to show that $C(t_1)$ is linearly independent over $\Z$, that is the only vanishing linear combination of pairwise distinct elements of $C(t_1)$ with integral coefficients is the trivial one.

Consider any vanishing linear combination
\begin{equation}\label{lin/eqn}
\textstyle\sum_{j=1}^m k_j\,\(t_1\)r_{i_{j,1}} r_{i_{j,2}} \dots r_{i_{j,n_j}} = 0
\end{equation}
of pairwise distinct conjugates $\(t_1\)r_{i_{j,1}} r_{i_{j,2}} \dots r_{i_{j,n_j}}$ of $t_1$, with coefficients $k_1, \dots, k_m \in \Z$.
Since $\(t_1\)r_1r_2r_3 = t_1$, without loss of generality we assume that all the $n_j$ are even.
Then, equation (\ref{phi/eqn}) tells us that the oriented angle from $t_1$ to $\(t_1\)r_{i_{j,1}} r_{i_{j,2}} \dots r_{i_{j,n_j}}$ equals
\begin{equation}\label{ang1/eqn}
2(\beta_{i_{j,2}} \!- \beta_{i_{j,1}}) + \dots + 2(\beta_{i_{j,n_j}} \!- \beta_{i_{j,n_j-1}}) \text{ mod } 2\pi\,,
\end{equation}
where $\beta_i$ denotes the oriented angle from $e_1$ to $e_i$, namely $\beta_1 = 0$, $\beta_2 = \pi - \alpha_3$ and $\beta_3 = \pi + \alpha_2$ mod $2\pi$.
This can also be written in the form
\begin{equation}\label{ang2/eqn}
m_{j,2} \alpha_2 + m_{j,3}\alpha_3 \text{ mod } 2\pi\,,
\end{equation}
with $m_{j,2}$ and $m_{j,3}$ even integers, hence the scalar product of (\ref{lin/eqn}) with $t_1$ gives
\begin{equation}\label{algrel1/eqn}
\textstyle\sum_{j=1}^m k_j \cos(m_{j,2} \alpha_2 + m_{j,3}\alpha_3) = 0\,,
\end{equation}
where $\rho^2$, the squared norm of $t_1$ (and all its conjugates), has been collected as a common factor and canceled.
Similarly, the scalar product of (\ref{lin/eqn}) with the vector obtained by rotating $t_1$ of $\pi/2$ radians gives
\begin{equation}\label{algrel2/eqn}
\textstyle\sum_{j=1}^m k_j \sin(m_{j,2} \alpha_2 + m_{j,3}\alpha_3) = 0\,.
\end{equation}

Notice that in the above equations the pairs $(m_{j,2}, m_{j,3})$ are different from each other, being the conjugates in (\ref{lin/eqn}) pairwise distinct. Moreover, possibly after suitable changes of signs in order to have either $m_{j,2} > 0$ or $m_{j,2} = 0$ and $m_{j,3} \geq 0$, we can collect the (at most two) terms corresponding to opposite pairs.

According to Lemma \ref{generic/thm}, the identities (\ref{algrel1/eqn}) and (\ref{algrel2/eqn}) hold for every triangle, that is for every $\alpha_2,\alpha_3 > 0$ such that $\alpha_2 + \alpha_3 < \pi$. Therefore, the linear independence over the reals of the complex functions $(x_1,x_2) \mapsto \exp(i(m_1 x_1 + m_2 x_2))$ with $m_1 > 0$ or $m_1 = 0$ and $m_2 \geq 0$, allows us to conclude that $k_j=0$ for every $j=1,\dots,m$.
\end{proof}

In view of the above proof, if $\tau$ is a generic triangle then any conjugate $t \in C(t_1)$ can be obtained from $t_1$ by a rotation of $2m\alpha_2 - 2n\alpha_3$ radians for some (uniquely determined, being $\tau$ typical) integers $m$ and $n$, hence $t = \(t_1\)(r_1r_2)^{n}(r_1r_3)^{m}$ according to equations (\ref{phi/eqn}) and (\ref{conjugate/eqn}). Therefore, for a generic triangle $\tau$ we can write
\begin{equation}\label{Ct/eqn}
C(t_1) = \{t_{n,m} = \(t_1\)(r_1r_2)^{n}(r_1r_3)^{m}\,,\,n,m \in \Z\}\,.
\end{equation}

Based on Theorem \ref{T-tau/thm}, next two Theorems \ref{G-tau/thm} and \ref{H-tau/thm} provide an infinite presentation of $G_\tau$ for a generic triangle $\tau$ and show that no such a finite presentation can exist.\break A minimal presentation of $G_\tau$ will be given in Section 5.

\medskip
\begin{theorem}\label{G-tau/thm}
For a generic triangle $\tau$ the group $G_\tau$ admits the presentation 
$$\langle x_1,x_2,x_3 \,|\, x_1^2, x_2^2, x_3^2, 
[w,\(w\)(x_1x_2)^n(x_1x_3)^m],\,n,m \in \Z\,,(n,m)\neq(0,0)\rangle\,,$$
with $w = (x_1x_2x_3)^2$ and the symbols $x_1,x_2,x_3$ corresponding to $r_1,r_2,r_3$, respectively. 
\end{theorem}

\begin{proof}
By a standard argument (see \cite[Sec. 10.2]{Joh97}), a presentation of $G_\tau$ can be derived from presentations of the groups $T_\tau$ and $S_\tau$ involved in the exact sequence (\ref{Gt-exact/eqn}). In view of Proposition \ref{typical/thm} a presentation of $S_\tau$ is given by Proposition \ref{S-tau/thm}, while $T_\tau$ is free abelian on the set of generators (\ref{Ct/eqn}), according to Theorem \ref{T-tau/thm}.

Pulling back the generators $s_i$ of $S_\tau$ to the generators $r_i$ of $G_\tau$, the relations $s_i^2 = 1$ still hold in the same form $r_i^2 = 1$, while the relation $(s_1s_2s_3)^2 = 1$ turns into the identity $(r_1r_2r_3)^2 = t_1 = t_{0,0}$. Moreover, for the generators of $T_\tau$ we have
\begin{equation}\label{tnm/eqn}
t_{n,m} = \(t_{0,0}\)(r_1r_2)^{n}(r_1r_3)^{m} = \((r_1r_2r_3)^2\)(r_1r_2)^{n}(r_1r_3)^{m}\,.
\end{equation}

At this point, to complete the set of relations for $G_\tau$ it remains to rewrite in the generators $r_i$'s, by using equation (\ref{tnm/eqn}), the commutators $[t_{n,m},t_{n',m'}]$ and the equations
\begin{equation}\label{tnmconj/eqn}
\(t_{n,m}\)r_1 = t_{-n+1,-m-1}\ ,\ 
\(t_{n,m}\)r_2 = t_{-n+2,-m-1}\ ,\ 
\(t_{n,m}\)r_3 = t_{-n+1,-m}\ ,
\end{equation}
which express in terms of the $t_{n,m}$'s their conjugates by the $r_i$'s. The latter equations could be easily shown to hold, by taking into account equation (\ref{conjugate/eqn}) and the commutativity of $\SO(2)$, and by using the relations $r_i^2$ and the trivial identity $\((r_1r_2r_3)^2\) r_1r_2r_3 = (r_1r_2r_3)^2$. However, we are going to validate them in a different way.

In fact, the rest of the proof is aimed to see how the rewriting of equations (\ref{tnmconj/eqn}) in the $r_i$'s, as well the rewriting of the commutators $[t_{n,m},t_{n',m'}]$ with $(n',m')\neq (n,m)$, can be derived from the relations $r_i^2$ and the relations
\begin{equation}\label{relGt/eqn}
[(r_1r_2r_3)^2,\((r_1r_2r_3)^2\)(r_1r_2)^n(r_1r_3)^m]
\end{equation}
with $n,m \in \Z$ and $(n,m) \neq (0,0)$, which represent the special commutators $[t_{0,0},t_{n,m}]$.

We start by observing that the relations $r_i^2$ imply
\begin{equation}\label{1312/eqn}
[r_1r_3,r_1r_2] = r_1r_3r_1r_2r_3r_1r_2r_1 = \((r_1r_2r_3)^2\)(r_1r_3)^{-1}\,,
\end{equation}
which in turn, together with the relation $[(r_1r_2r_3)^2,\((r_1r_2r_3)^2\)(r_1r_2)^n(r_1r_3)^m]$ conjugated by $(r_1r_3)^{-1}$, implies 
\begin{equation}\nonumber
\((r_1r_2r_3)^2\)(r_1r_2)^n(r_1r_3)^mr_1r_2 = \((r_1r_2r_3)^2\)(r_1r_2)^n(r_1r_3)^{m\mp1}r_1r_2(r_1r_3)^{\pm1}\,.
\end{equation}
Hence, by increasing/decreasing induction on $m$, based on the trivial case of $m=0$,
\begin{equation}\nonumber
\((r_1r_2r_3)^2\)(r_1r_2)^n(r_1r_3)^m(r_1r_2)^{\pm1} = \((r_1r_2r_3)^2\)(r_1r_2)^{n\pm1}(r_1r_3)^m
\end{equation}
for every $m \in \Z$.
Finally, by increasing/decreasing induction on $n$, based on the trivial case of $n=0$, we get
\begin{equation}\label{13m12n/eqn}
\((r_1r_2r_3)^2\)(r_1r_3)^m(r_1r_2)^n = \((r_1r_2r_3)^2\)(r_1r_2)^n(r_1r_3)^m
\end{equation}
for every $n,m \in \Z$.

As a consequence of (\ref{13m12n/eqn}), the rewriting of any commutator $[t_{n,m},t_{n',m'}]$ is equivalent up to conjugation to that of $[t_{0,0},t_{n'-n,m'-m}]$. In fact,
\begin{equation}\nonumber
[\mkern1mu\((r_1r_2r_3)^2\)(r_1r_2)^n(r_1r_3)^m,\((r_1r_2r_3)^2\)(r_1r_2)^{n'}(r_1r_3)^{m'}]
\end{equation}
once conjugated by $(r_1r_3)^{-m}(r_1r_2)^{-n}$ becomes
\begin{equation}\nonumber
[(r_1r_2r_3)^2,\((r_1r_2r_3)^2\)(r_1r_2)^{n'}(r_1r_3)^{m'-m}(r_1r_2)^{-n}]\,,
\end{equation}
and this is equivalent to
\begin{equation}\nonumber
[(r_1r_2r_3)^2,\((r_1r_2r_3)^2\)(r_1r_2)^{n'-n}(r_1r_3)^{m'-m}]
\end{equation}
by equation (\ref{13m12n/eqn}).

Moreover, we obtain the rewriting of the relations (\ref{tnmconj/eqn}) from the relations $r_i^2$ and the relations (\ref{relGt/eqn}), by the following chains of equalities, whose last step is based on two applications of equation (\ref{13m12n/eqn}):
\begin{equation}\nonumber
\begin{array}{rcl}\vrule width0pt depth0pt height12pt
\((r_1r_2r_3)^2\)(r_1r_2)^n(r_1r_3)^mr_1
&=& \((r_1r_2r_3)^2\)r_1(r_1r_2)^{-n}(r_1r_3)^{-m}\\[2pt]
&=& \((r_1r_2r_3)^2\)r_1r_2r_3r_1(r_1r_2)^{-n}(r_1r_3)^{-m}\\[2pt]
&=& \((r_1r_2r_3)^2\)(r_1r_2)^{-n+1}(r_1r_3)^{-m-1}\,;\\[10pt]
\((r_1r_2r_3)^2\)(r_1r_2)^n(r_1r_3)^mr_2
&=& \((r_1r_2r_3)^2\)r_1(r_1r_2)^{-n}(r_1r_3)^{-m}r_1r_2\\[2pt]
&=& \((r_1r_2r_3)^2\)r_1r_2r_3r_1(r_1r_2)^{-n}(r_1r_3)^{-m}r_1r_2\\[2pt]
&=& \((r_1r_2r_3)^2\)(r_1r_2)^{-n+2}(r_1r_3)^{-m-1}\,;\\[10pt]
\((r_1r_2r_3)^2\)(r_1r_2)^n(r_1r_3)^mr_3
&=& \((r_1r_2r_3)^2\)r_1(r_1r_2)^{-n}(r_1r_3)^{-m}r_1r_3\\[2pt]
&=& \((r_1r_2r_3)^2\)r_1r_2r_3r_1(r_1r_2)^{-n}(r_1r_3)^{-m}r_1r_3\\[2pt]
&=& \((r_1r_2r_3)^2\)(r_1r_2)^{-n+1}(r_1r_3)^{-m}\,.
\end{array}
\end{equation}

This concludes the proof.
\end{proof}


\section{The rotation subgroup\label{rotations/sec}}

By intersecting the exact sequence (\ref{Gt-exact/eqn}) with the group $\SE(2)$ of orientation preserving Euclidean isometries, we get the new exact sequence 
\begin{equation}\label{Ht-exact/eqn}
1 \longrightarrow T_\tau \buildrel \iota_{\tau\smash{|}} \over \longrightarrow H_\tau \buildrel \lambda_{\tau\smash{|}} \over \longrightarrow R_\tau \longrightarrow 1\,,
\end{equation}
where $H_\tau = G_\tau \cap \SE(2)$ is a subgroup of index 2 in $G_\tau$, which we call the \textsl{rotation subgroup}, while $R_\tau = S_\tau \cap \SE(2) \subset \SO(2)$
is the abelian subgroup of linear rotations in $S_\tau$. As an extension of an abelian group by another abelian group, $H_\tau$ is metabelian, hence $G_\tau$ is virtually metabelian.

\pagebreak

If $\tau$ is a typical triangle, then $R_\tau$ is a dense free abelian subgroup of $\SO(2)$ generated by the two rotations $s_1s_2$ and $s_1s_3$ (cf. proof of Lemma \ref{stable/thm}).
Moreover, in the light of the first part of Theorem \ref{T-tau/thm}, equations (\ref{Ct/eqn}) and (\ref{1312/eqn}) imply that $T_\tau$ coincides with the commutator subgroup $[H_\tau,H_\tau]$ when $\tau$ is typical.

\begin{theorem}\label{H-tau/thm}
For a generic triangle $\tau$ the rotation subgroup $H_\tau \subset G_\tau$ is a free metabelian group of rank 2 generated by the rotations $r_1r_2$ and $r_1r_3$, meaning that it is isomorphic to the metabelianization $F_2/[\mkern1mu[F_2,F_2],[F_2,F_2]\mkern1mu]$ of the free group $F_2$ on two generators corresponding to those rotations. As a consequence, $\,G_\tau$ does not admit any finite presentation.
\end{theorem}

\begin{proof}
Since $F_2/[\mkern1mu[F_2,F_2],[F_2,F_2]\mkern1mu]$ is known not to admit any finite presentation \cite{Shm65} (cf. also \cite{BGH13}) and $H_\tau$ is a finite index subgroup of $G_\tau$, the second part of the statement follows (see \cite[Sec. 9.1]{Joh97}) once we know that $H_\tau \cong F_2/[\mkern1mu[F_2,F_2],[F_2,F_2]\mkern1mu]$. 

The above observation that $[H_\tau,H_\tau] = T_\tau$ together with the second part of Theorem \ref{T-tau/thm} and the fact that $S_\tau$ is free abelian of rank 2, would suffice to prove that $H_\tau$ is free metabelian of rank 2. However, for future reference, we give an explicit isomorphism $H_\tau \cong F_2/[\mkern1mu[F_2,F_2],[F_2,F_2]\mkern1mu]$ by way of a presentation of $H_\tau$.

Starting from the presentation of $G_\tau$ given in Theorem \ref{G-tau/thm} and applying the Reide\-meister-Schreier method (see \cite[Sec. 9.1]{Joh97}) to the subgroup $H_\tau$ with Schreier transversal $\{1,x_1\}$, we get the following presentation for $H_\tau$
$$\langle y_2,y_3,z_1,z_2,z_3 \,|\, z_1, y_2z_2, y_3z_3, a_{n,m}, z_2y_2, z_3y_3, a'_{n,m},\, n,m \in \Z\,,(n,m)\neq(0,0)\rangle\,,$$
where the generators are given by $y_2 = x_2x_1^{-1}, y_3 = x_3x_1^{-1}, z_1 = x_1^2, z_2 = x_1x_2, z_3 = x_1x_3$, and the relations $a_{n,m}$ and $a'_{n,m}$ are the transcriptions in term of such generators of the commutator $[w,\(w\)(x_1x_2)^n(x_1x_3)^m]$ and its conjugate $\([w,\(w\)(x_1x_2)^n(x_1x_3)^m]\)x_1^{-1} = [\(w\)x_1z_1^{-1},\(w\)x_1(x_2x_1)^n(x_3x_1)^mz_1^{-1}]$, respectively, that is
\begin{equation}\nonumber
\begin{array}{c}\vrule width0pt depth0pt height6pt
a_{n,m} = [z_2y_3z_1y_2z_3,\(z_2y_3z_1y_2z_3\)z_2^nz_3^m]\,,\\[6pt]
a'_{n,m} = [\(y_2z_3z_2y_3z_1\)z_1^{-1},\(y_2z_3z_2y_3z_1\)(y_2z_1)^n(y_3z_1)^mz_1^{-1}]\,.
\end{array}
\end{equation}

Now, we eliminate the generators $z_1,z_2$ and $z_3$, by using the relations $z_1, y_2z_2$ and $y_3z_3$, and then replace $y_3$ by $y_3^{-1}$ to obtain the new presentation for $H_\tau$
\begin{equation}\nonumber
\langle y_2,y_3\,|\,[\mkern1mu[y_2,y_3],\([y_2,y_3]\)y_2^ny_3^m\mkern1mu],[\mkern1mu[y_2^{-1},y_3^{-1}],\([y_2^{-1},y_3^{-1}]\)y_2^ny_3^m\mkern1mu],\, n,m \in \Z\,,(n,m)\neq(0,0)\rangle\,.
\end{equation}

Finally, since $H_\tau$ is metabelian and all the relations in the above presentation of $H_\tau$ belong to $[\mkern1mu[F_2,F_2],[F_2,F_2]\mkern1mu]$, we can conclude that $H_\tau \cong F_2/[\mkern1mu[F_2,F_2],[F_2,F_2]\mkern1mu]$ with the rotations $r_1r_2$ and $r_1r_3$ corresponding to the free generators $y_2^{-1}$ and $y_3$ of $F_2$,\break respectively.
\end{proof}

In the following, we provide presentations of the groups $H_\tau$ and $G_\tau$, which are minimal in the sense that no relation can be removed without changing the group. According to Theorem \ref{H-tau/thm}, this solves in the affermative the problem of the existence of a minimal presentation of the free metabelian group $F_2/[\mkern1mu[F_2,F_2],[F_2,F_2]\mkern1mu]$ of rank 2 (cf. \cite{BCGS14}).

The two presentations are given in Theorems \ref{H-tau-min/thm} and \ref{G-tau-min/thm}, respectively, while their minimality is proved in Theorem \ref{minimality/thm}. For $H_\tau$ we just refine the presentation considered in the proof of Theorem \ref{H-tau/thm} to make it minimal. On the contrary, the minimal presentation of $G_\tau$ is derived from that of $H_\tau$, and it has different relations with respect to the one given in Theorem \ref{G-tau/thm}.

\begin{theorem}\label{H-tau-min/thm}
For a generic triangle $\tau$ the group $H_\tau$ admits the presentation 
$$\langle y_2,y_3 \,|\,  
[\mkern1mu[y_2,y_3],\([y_2,y_3]\)y_2^ny_3^m]\,,\,n,m \in \Z\,,(n,m) > (0,0)\rangle\,,$$
with the symbols $y_2,y_3$ corresponding to the rotations $r_2r_1$ and $r_1r_3$, respectively. 
\end{theorem}

\begin{proof}
The relations in the statement imply that
\begin{equation}\nonumber
\([y_2,y_3]\)y_2^ny_3^m y_2 = \([y_2,y_3]\)y_2^ny_3^{m\mp1} y_2 y_3^{\pm1}
\end{equation}
for every $(n,m) > (0,0)$. From this family of equalities, arguing as in the proof of Theorem \ref{H-tau/thm} when obtaining equation (\ref{13m12n/eqn}), by increasing/decreasing induction on $m$, based on the trivial case of $m = 0$, and then by induction on $n \geq 0$, based on the trivial case of $n = 0$, we get 
\begin{equation}\label{y3my2n/eqn}
\([y_2,y_3]\)y_3^my_2^n = \([y_2,y_3]\)y_2^ny_3^m
\end{equation}
for every $(n,m) > (0,0)$.
By using such equation, we can rewrite $[\mkern1mu[y_2,y_3],\([y_2,y_3]\)y_2^ny_3^m]$ as 
$[\mkern1mu[y_2,y_3],\([y_2,y_3]\)y_3^my_2^n]$, which once conjugated by $y_2^{-n}y_3^{-m}$ and inverted becomes
\begin{equation}\nonumber
[\mkern1mu[y_2,y_3],\([y_2,y_3]\)y_2^{-n}y_3^{-m}]\,.
\end{equation}
This means that the relation $[\mkern1mu[y_2,y_3],\([y_2,y_3]\)y_2^ny_3^m]$, assumed to hold for every $(n,m) > (0,0)$, actually holds for every $(n,m) \neq (0,0)$.

Now, taking into account the identity $[y_2^{-1},y_3^{-1}] = \([y_2,y_3]\)y_2y_3$, from (\ref{y3my2n/eqn}) we get
\begin{equation}\label{y3my2ninv/eqn}
\([y_2^{-1},y_3^{-1}]\)y_3^my_2^n = \([y_2^{-1},y_3^{-1}]\)y_2^ny_3^m
\end{equation}
for every $(n,m) > (0,0)$, and we
can write the commutator $[\mkern1mu[y_2^{-1},y_3^{-1}],\([y_2^{-1},y_3^{-1}]\)y_2^ny_3^m]$ with as 
 $[\([y_2,y_3]\)y_2y_3,\([y_2,y_3]\)y_2y_3y_2^ny_3^m]$. By using equation (\ref{y3my2n/eqn}) once again, we immediately have the validity of the relation $[\mkern1mu[y_2^{-1},y_3^{-1}],\([y_2^{-1},y_3^{-1}]\)y_2^ny_3^m]$ for every $(n,m) > (0,0)$, and then we can extend such validity to every $(n,m) \neq (0,0)$, as we did above for the relation $[\mkern1mu[y_2,y_3],\([y_2,y_3]\)y_2^ny_3^m]$, but using equation (\ref{y3my2ninv/eqn}) instead of (\ref{y3my2n/eqn}).
 
In conclusion, the presentation in the statement is equivalent to the last presentation of $H_\tau$ given in the proof of Theorem \ref{H-tau/thm} (after the replacement of $y_3$ by $y_3^{-1}$), whose generators $y_2 = x_2x_1^{-1}$ and $y_3 = (x_3x_1^{-1})^{-1}$ correspond to the rotations $r_2r_1$ and $r_1r_3$, respectively.
\end{proof}

\begin{theorem}\label{G-tau-min/thm}
For a generic triangle $\tau$ the group $G_\tau$ admits the presentation 
$$\langle x_1,x_2,x_3 \,|\, x_1^2, x_2^2, x_3^2, 
[v,\(v\)(x_2x_1)^n(x_1x_3)^m]\,,\,n,m \in \Z\,,(n,m)>(0,0)\rangle\,,$$
with $v = [x_2x_1,x_1x_3]$ and the symbols $x_1,x_2,x_3$ corresponding to $r_1,r_2,r_3$, respectively. 
\end{theorem}

\begin{proof}
We think of $G_\tau$ as an extension of $\Z_2$ by $H_\tau$ and deduce the presentation of it from that of $H_\tau$ given by Theorem \ref{H-tau-min/thm} and the obvious one of $\Z_2$ in the usual way (see \cite[Sec. 10.2]{Joh97}).

As the generators we have $x_1$, corresponding to the reflection $r_1$ (a lifting to $G_\tau$ of the generator of $\Z_2$), and the generators $y_2$ and $y_3$ in the presentation of $H_\tau$, corresponding to the rotations $r_2r_1$ and $r_1r_3$, respectively.

As the relation, besides the ones in the presentation of $H_\tau$, we have
\begin{equation}\label{Gt-sqrel/eqn}
x_1^2\ ,\ x_1y_2x_1^{-1} = y_2^{-1}\ ,\ x_1y_3x_1^{-1} = y_3^{-1}\,,
\end{equation}
the first of which comes from the relation of $\Z_2$, while the others express in terms of the generators $y_2$ and $y_3$ their conjugates by $x_1$.

In the new generators $x_1,x_2,x_3$, with $x_2 = y_2x_1$ and $x_3 = x_1y_3$ corresponding to the reflection $r_2$ and $r_3$, respectively, the relations (\ref{Gt-sqrel/eqn}) reduce to $x_1^2,x_2^2,x_3^2\,$, and up to such relations those in the presentation of $H_\tau$ read as $[\mkern1mu[x_2x_1,x_1x_3],\([x_2x_1,x_1x_3]\)(x_2x_1)^n(x_1x_3)^m]$ for every $(n,m) > (0,0)$.
\end{proof}

\begin{theorem}\label{minimality/thm}
The presentations given in Theorems \ref{H-tau-min/thm} and \ref{G-tau-min/thm} are minimal.
\end{theorem}

\begin{proof}
We first prove the minimality of the presentation of $H_\tau$, then we see how the argument can be adapted to prove the minimality of the presentation of $G_\tau$. In both cases, the idea is to apply the Reidemeister-Schreier method (see \cite[Sec. 9.1]{Joh97}) to obtain a presentation of the subgroup $T_\tau$ induced by the presentation of $H_\tau$ (resp. $G_\tau$), and show that if any single relation is removed from this last presentation, then the presentation induced on $T_\tau$ would give a non commutative group.

We start with the presentation of $H_\tau$ in Theorem \ref{H-tau-min/thm}, and we choose as a Schreier transversal for $T_\tau \subset H_\tau$ the set $\{y_2^iy_3^j\,,\,i,j \in \Z\},$ where $y_2^iy_3^j$ corresponds to the lifting $(r_2r_1)^i(r_1r_3)^j \in H_\tau$ of the generic rotation $(s_2s_1)^i(s_1s_3)^j \in R_\tau$ in the exact sequence (\ref{Ht-exact/eqn}).

According to this choice, taking into account that $\lambda_\tau(y_2^iy_3^j y_2) = \lambda_\tau(y_2^{i+1}y_3^j)$, being $R_\tau \subset \text{SO}(2)$ abelian, and that $y_2^iy_3^j y_3 = y_2^iy_3^{j+1}$, a set of generators for $T_\tau$ is
\begin{equation}
\{b_{i,j} = y_2^iy_3^j y_2y_3^{-j}y_2^{-i-1} = [y_2^iy_3^j,y_2]\,,\,i,j \in \Z\,,j\neq 0\}\,.
\end{equation}
Moreover, the relations are the transcription in the $b_{i,j}$'s of the words
\begin{equation}\label{bnmkl/eqn}
c_{n,m,k,\ell}=[\mkern1mu\([y_2,y_3]\)y_3^\ell y_2^k,\([y_2,y_3]\)y_2^ny_3^m y_3^\ell y_2^k\mkern1mu]\,,
\end{equation}
with $n,m,k,\ell \in \Z$ and $(n,m) > (0,0)$.

In order to carry out the transcription, we observe that
\begin{equation}\label{y2iy3j/eqn}
[y_2^i,y_3^j] = \cases{\!b_{i,j}b_{i+1,j}\dots b_{-1,j}\hfill \text{ if } i \leq 0\cr 
\!(b_{0,j} b_{1,j} \dots b_{i-1,j})^{-1} \hfill \text{ if } i \geq 0}\,,
\end{equation}
where we put $b_{i,0} = 1$ for every $i \in \Z\,$.
In fact, a part from the trivial case of $i = 0$, for $i=\pm1$ we immediately have $[y_2,y_3^j] = b_{0,j}^{-1}$ and $[y_2^{-1},y_3^j] = b_{-1,j}$. 
Then, decreasing induction on $i \leq -1$ gives
$$[y_2^i,y_3^j] = y_2^iy_3^j y_2 y_3^{-j}y_2^{-i} y_2^{-1} y_2^{i+1}y_3^j y_2^{-i-1}y_3^{-j} =
b_{i,j}[y_2^{i+1},y_3^j] = b_{i,j}b_{i+1,j}\dots b_{-1,j}\,,$$
while increasing induction on $i \geq 1$ gives
$$[y_2^i,y_3^j] = y_2 y_2^{i-1}y_3^j y_2^{-1} y_3^{-j}y_2^{-i+1}y_2^{i-1}y_3^j y_2^{-i+1}y_3^{-j} =
b_{i-1,j}^{-1}[y_2^{i-1},y_3^j] = b_{i-1,j}^{-1}b_{i-2,j}^{-1}\dots b_{0,j}^{-1}\,.$$

Now, direct inspection shows that
$$\([y_2,y_3]\)y_3^\ell y_2^k = [y_2^{-k},y_3^{-\ell}][y_3^{-\ell},y_2^{-k+1}][y_2^{-k+1},y_3^{-\ell+1}][y_3^{-\ell+1},y_2^{-k}]\,,$$
and after performing the replacements (\ref{y2iy3j/eqn}), separately for the two cases $k \leq 0$ and $k \geq 1$, we get in both cases
\begin{equation}\label{T-tau-rel1/eqn}
\([y_2,y_3]\)y_3^\ell y_2^k = b_{-k,-\ell} b_{-k,-\ell+1}^{-1}\,.
\end{equation}
Analogously, direct inspection shows that
$$
\begin{array}{rcl}\vrule width0pt depth0pt height12pt
\([y_2,y_3]\)y_2^ny_3^my_3^\ell y_2^k &=& 
[y_2^{-k},y_3^{-m-\ell}][y_3^{-m-\ell},y_2^{-k-n+1}][y_2^{-k-n+1},y_3^{-m-\ell+1}]\cr
&& {\ } \cdot 
[y_3^{-m-\ell+1},y_2^{-k-n}][y_2^{-k-n},y_3^{-m-\ell}][y_3^{-m-\ell},y_2^{-k}]\,,
\end{array}
$$
and after performing the replacements (\ref{y2iy3j/eqn}), separately for the two cases $k \leq 0$ and $k \geq 1$ if $n = 0$ and for the three cases $k \leq -n\,, -n < k < 0$ and $k \geq 0$ if $n > 0$, we get in all cases
\begin{equation}\label{T-tau-rel2/eqn}
\([y_2,y_3]\)y_2^ny_3^my_3^\ell y_2^k = \(b_{-n-k,-m-\ell}b_{-n-k,-m-\ell+1}^{-1}\)
\mathop{\raise-3pt\hbox{\Large$\Pi$}}\nolimits_{i=0}^{n-1}\! b_{-n-k+i,-m-\ell}\,.
\end{equation}
Based on (\ref{T-tau-rel1/eqn}) and (\ref{T-tau-rel2/eqn}), we have the transcription
\begin{equation}\label{bnmkl-trans/eqn}
c_{n,m,k,\ell} = 
[\,c_{-k,-\ell} c_{-k,-\ell+1}^{-1}\,,\,\(b_{-n-k,-m-\ell}b_{-n-k,-m-\ell+1}^{-1}\)
\mathop{\raise-3pt\hbox{\Large$\Pi$}}\nolimits_{i=0}^{n-1}\! b_{-n-k+i,-m-\ell}\,]\,,
\end{equation}\label{T-tau-rel/eqn}
for every $n,m,k,\ell \in \Z$ and $(n,m) > (0,0)$.

At this point, we consider the new set of generators for $T_\tau$
\begin{equation}\label{dij/eqn}
\{d_{i,j} = b_{-i,-j} b_{-i,-j+1}^{-1}\,,\,i,j \in \Z\}\,,
\end{equation}
where $b_{i,0} = 1$ for every $i \in \Z$, and hence $d_{i,1} = b_{-i,-1}$ and $d_{i,0} = b_{-i,1}^{-1}$. Starting from the last two equalities, and proceeding by decreasing induction on $j \leq -1$ and by increasing induction on $j \geq 1$, we obtain
\begin{equation}\nonumber
b_{i,j} = \cases{\!d_{-i,-j}d_{-i,-j-1}\dots d_{-i,1}\hfill \text{ if } j \leq 0\cr 
\!(d_{-i,0}d_{-i,-1} \dots d_{-i,-j+1})^{-1} \hfill \text{ if } j \geq 0}\,,
\end{equation}
with $b_{i,0}$ corresponding to the empty word in the $d_{i,j}$'s in both cases.

Then, by performing these replacements in the equation (\ref{bnmkl-trans/eqn}), we obtain as the set of relations for $T_\tau$ in the generators $d_{i,j}$
\begin{equation}\label{dnmkl/eqn}
\{e_{n,m,k,\ell} = [d_{k,\ell},\(d_{n+k,m+\ell}\)u_{n,m,k,\ell}]\,,\, n,m,k,\ell \in \Z\,,
(n,m)>(0,0)\}\,,
\end{equation}
where $u_{n,m,k,\ell}$ is a certain word in the $d_{i,j}$'s.

Now, assume by contradiction that a single relation $[\mkern1mu[y_2,y_3],\([y_2,y_3]\)y_2^{n_0}y_3^{m_0}]$ with $(m_0,n_0) > (0,0)$ can be removed from the presentation of $H_\tau$, in such a way that we still have a presentation of $H_\tau$. Then, in the set of relations (\ref{dnmkl/eqn}) for the induced presentation of $T_\tau$ with generators $d_{i,j}$, all the relations $e_{n_0,m_0,k,\ell}$ with $k,\ell \in \Z$ are omitted.
This allows us to define a homomorphism $T_\tau \to \Sigma_3$ that sends $d_{0,0}$ and $d_{n_0,m_0}$ to the transpositions $(1\,2)$ and $(2\,3)$, respectively, and any other $d_{i,j}$ to the identity, in contrast with the fact that $T_\tau$ is abelian. In fact, by replacing the generators by the corresponding transpositions in the relation $e_{n,m,k,\ell}$, we always get the identity if at least one of $(k,\ell)$ and $(n+k,m+\ell)$ does not coincide with $(0,0)$ or $(n_0,m_0)\,$. Therefore, being $(n,m),(n_0,m_0)>(0,0)\,$, the only possibility for not having the identity is $(k,\ell) = (0,0)$ and $(n,m) = (n_0,m_0)$. But this cannot happen since the relation $e_{n_0,m_0,0,0}$ is missing.

In order to prove the minimality of the presentation of $G_\tau$ in Theorem \ref{G-tau-min/thm}, we first apply the Reidemeister-Schreier method to derive from it a presentation of the subgroup $H_\tau \subset G_\tau$ with Schreier transversal $\{1, x_1\}$. Arguing as in the proof of Theorem \ref{H-tau/thm}, with $v$ and $\(v\)(x_2x_1)^n(x_1x_3)^m$ in place of $w$ and $\(w\)(x_1x_2)^n(x_1x_3)^m$, respectively, we obtain for $H_\tau$ the presentation
\begin{equation}\nonumber
\langle y_2,y_3\,|\,[\mkern1mu[y_2,y_3],\([y_2,y_3]\)y_2^ny_3^m\mkern1mu],[\mkern1mu[y_2^{-1},y_3^{-1}],\([y_2^{-1},y_3^{-1}]\)y_2^ny_3^m\mkern1mu],\, n,m \in \Z\,,(n,m)>(0,0)\rangle\,.
\end{equation}
Then, we perform once again the Reidemeister-Schreier method on this presentation to get a presentation of the subgroup $T_\tau \subset H_\tau$ with Schreier transversal $\{y_2^iy_3^j\,,\,i,j \in \Z\}$. Similar computations as in the first part of this proof leads a presentation having the same set of generators (\ref{dij/eqn}) and relations
\begin{equation}\label{ddnmkl/eqn}
\begin{array}{c}\vrule width0pt depth0pt height6pt
e_{n,m,k,\ell} = [d_{k,\ell},\(d_{n+k,m+\ell}\)u_{n,m,k,\ell}\,]\,,\\[6pt]
e'_{n,m,k,\ell} = [\(d_{k+1,\ell+1}\)u'_{k,\ell},\(d_{-n+k+1,-m+\ell+1}\)u''_{n,m,k,\ell}\,]\,,
\end{array}
\end{equation}
where $e_{n,m,k,\ell}$ is as in (\ref{dnmkl/eqn}), while $u'_{k,\ell}$ and $u''_{n,m,k,\ell}$ are suitable words in the $d_{i,j}$'s.

At this point the minimality of the presentation of $G_\tau$ can be deduced by the same argument
used above for $H_\tau$, based on the equality between the differences of indices $(k+1,\ell+1) - (-n+k+1,-m+\ell+1) = (n+k,m+\ell) - (k,\ell) = (n,m)\,$ in (\ref{ddnmkl/eqn}).
\end{proof}



\section{Examples of non-generic relations\label{examples/sec}}

As we have seen in the previous sections, apart from the obvious involutive property of the $r_i$'s, the only generic relations in $G_\tau$, that is the ones holding for $\tau$ a generic triangle or equivalently for every triangle $\tau$ (by Lemma \ref{generic/thm}), are the commutators of the translations in the free abelian subgroup $T_\tau$ generated by the conjugates of $t_1$.

Here, we briefly discuss the existence of extra non-generic relations for the subgroup $T_\tau$, and hence for the group $G_\tau$, in the case when the triangle $\tau$ is typical but not generic. In this respect, typical triangles are expected to present a rich unexplored structure, in some sense complementary to the one encoded by the relations $(r_1r_2)^{n_3}, (r_2r_3)^{n_1}$ and $(r_3r_1)^{n_2}$ for a rational triangle $\tau$ having angles $m_i \pi/n_i$ with $(m_i,n_i) = 1$, which has been widely considered in the literature after the pioneering work of Coxeter \cite{Co34}.

\begin{Figure}[b]{figure3/fig}
\vskip6pt
\fig{}{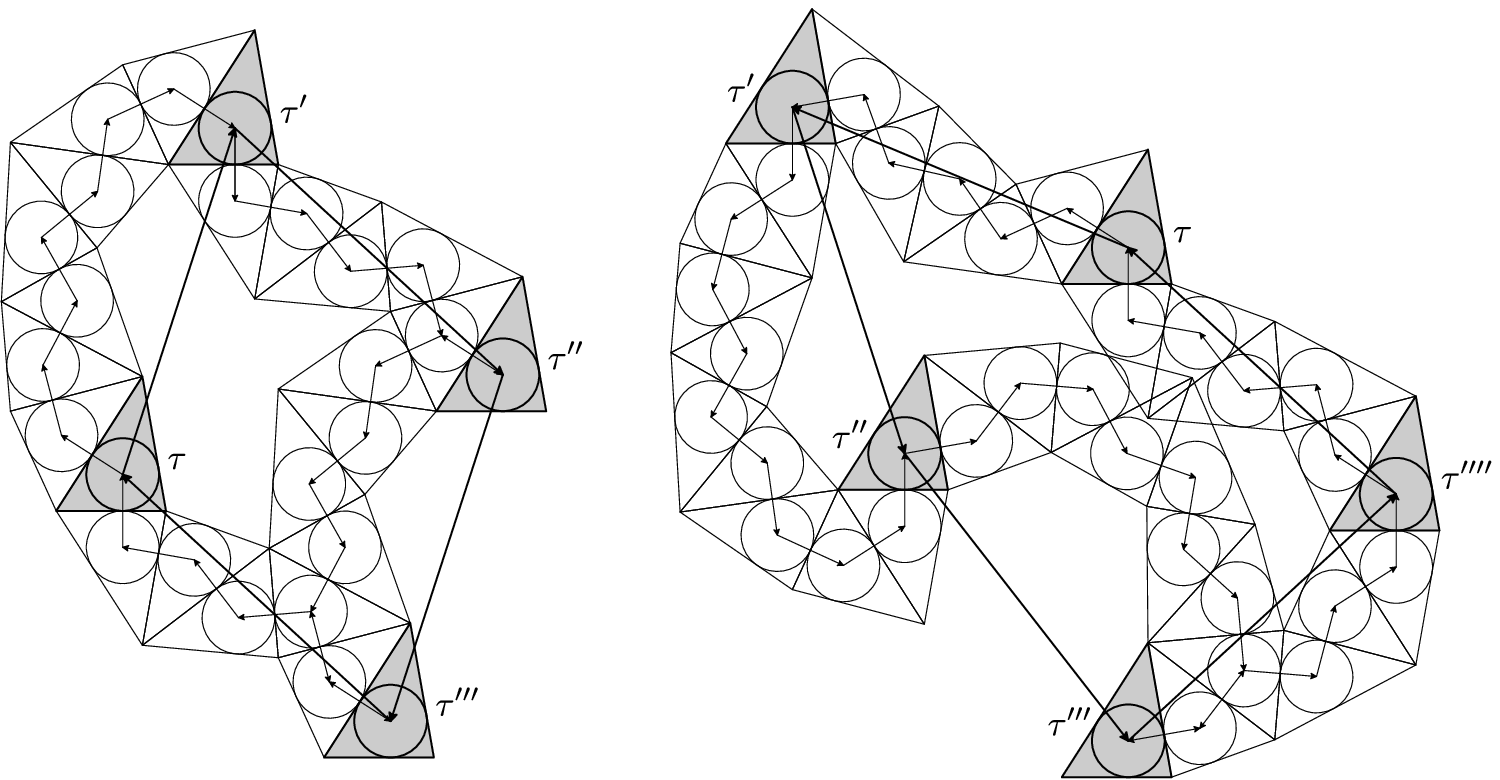}
    {}{Generic and non-generic relations}
\end{Figure}

In Figure \ref{figure3/fig} two relations of $G_\tau$ are represented in terms of the corresponding chain of triangles generated by each next reflection in the word, starting from $\tau$ and ending back to $\tau$. Namely, on the left side there is the generic relation given by the commutator
\begin{equation}\nonumber
\begin{array}{rcl}\vrule width0pt depth0pt height12pt
[t_1,\(t_1\)r_1r_3]
&=& t_1r_3r_1t_1r_1r_3t_1^{-1}r_3r_1t_1^{-1}r_1r_3\\[2pt]
&=& r_1r_2r_3r_1r_3r_1r_2r_3r_1r_2r_1r_3r_2r_1r_3r_1r_3r_2r_1r_3r_2r_3\,,
\end{array}
\end{equation}
where some $r_i^2$ has been canceled in the last expression, while on the right side there is the non-generic relation
\begin{equation}\nonumber
\begin{array}{l}\vrule width0pt depth0pt height12pt
t_1 \cdot \(t_1^{-1}\)r_1 \cdot \(t_1\)r_1r_3r_2 \cdot \(t_1^{-1}\)r_1r_3r_1 
\cdot \(t_1^{-1}\)r_3\\[2pt]
\kern20pt = r_1r_2r_3r_1r_2r_3r_1r_3r_2r_1r_2r_3r_1r_2r_3r_1r_3
r_2r_1r_3r_1r_3r_2r_1r_3r_2r_3r_1r_2r_1r_3r_2r_1r_3\,,
\end{array}
\end{equation}
which holds only for the triangles $\tau$ whose angles $\alpha_i$ satisfy a specific condition (in particular for all the triangles such that $2 \cos(2\alpha_2 + 2\alpha_3) - 2 \cos 2\alpha_2 = 1$).

The big vectors superposed to the chains of triangles in the figure correspond to the expression of the relation as a word in the set $C(t_1)$ of generators of the translation subgroup $T_\tau$, while the small vectors indicate the displacement of the incenter of the triangle under the action of each next reflection in the expression of the relation as a word in the generators $r_1,r_2$ and $r_3$ of $G_\tau$. The lengths of these two word representations of a relation, in the generators of $T_\tau$ and $G_\tau$ respectively, provide relatively independent measures of its complexity. 

The following table reports the number of stable words in the generators $r_1,r_2$ and $r_3$ of $G_\tau$ up to length 24, which are cyclically reduced with respect to the cancellation of the $r_i^2$'s, and pairwise distinct up to permutation of indices, inversion, conjugation and commutation of stable words. These have been obtained by a computer procedure in three steps: first, the generation of a complete list of all the cyclically reduced stable words of a given length; then, the elimination of duplicates up to permutation of indices, inversion and cyclic permutations of the word; finally, the detection of the remaining pairs of words (even of different lengths) equivalent up to conjugation and commutation of stable words, by comparing their expressions as linear combinations of vectors in $C(t_1)$. The total number of words of each length from 6 to 24 with respect to the generators $r_1,r_2$ and $r_3$ of $G_\tau$ in the last column, is subdivided in the previous columns according to the length from 1 to 12 with respect to the set $C(t_1)$ of generators of $T_\tau$.

\bigskip
\centerline{\includegraphics{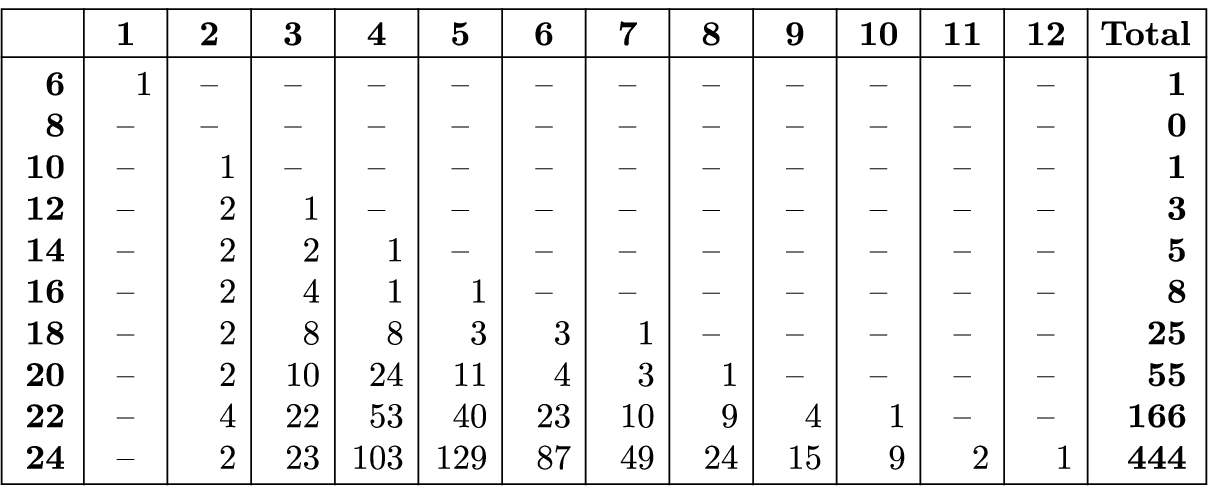}}
\bigskip

Now, in order to determine when a cyclically reduced stable word $r_{i_1} r_{i_2}\dots r_{i_n}$ represents the identity in $T_\tau$, one could directly observe that, according to Lemma \ref{stable/thm} and equation (\ref{ri/eqn}), this happens if and only if
\begin{equation}\label{rel1/eqn}
\textstyle\sum_{j=1}^n (v_{i_j})s_{i_{j+1}} \dots s_{i_n} = 0\,.
\end{equation}
Notice that the $j$-th term of this summation coincides with half the displacement vector of the incenter of the triangle under the action of the $j$-th reflection in $r_{i_1} r_{i_2}\dots r_{i_n}$.
Taking into account that all these vectors have the same norm, equation (\ref{phi/eqn}) could be applied to rewrite equation (\ref{rel1/eqn}) as a condition on the angles $\alpha_2$ and $\alpha_3$ of $\tau$ under which $r_{i_1} r_{i_2} \dots r_{i_n}$ is a relation for $T_\tau$.

A more convenient approach to the same condition on the angles of $T_\tau$ is provided by the proof of Theorem \ref{T-tau/thm}. Once the translation vector corresponding to the word $r_{i_1} r_{i_2} \dots r_{i_n}$ has been expressed as a linear combination of vectors in $C(t_1)$, equation\break (\ref{rel1/eqn}) can be put in the form
\begin{equation}\nonumber
\textstyle\sum_{j=1}^m k_j\,\(t_1\)r_{i_{j,1}} r_{i_{j,2}} \dots r_{i_{j,n_j}} = 0\,.
\end{equation}
Then, according to equations (\ref{ang1/eqn}) and (\ref{ang2/eqn}), we get the equivalent system
\begin{equation}\label{sys/eqn}
\left\{
\begin{array}{l}
\textstyle\sum_{j=1}^m k_j \cos(m_{j,2} \alpha_2 + m_{j,3}\alpha_3) = 0\\[6pt]
\textstyle\sum_{j=1}^m k_j \sin(m_{j,2} \alpha_2 + m_{j,3}\alpha_3) = 0
\end{array}
\right.,
\end{equation}
where $m_{j,2} \alpha_2 + m_{j,3}\alpha_3$ is the oriented angle from $t_1$ to $(t_1)r_{i_{j,1}} r_{i_{j,2}} \dots r_{i_{j,n_j}}$.

By a systematic computer search among the stable words up to length 24 generated as said above, we found that the shortest words in the $r_i$'s giving non-generic relations for some typical triangle have length 18. Up to permutation of indices, inversion, conjugation and commutation of stable
words, there are two of such words of length 18. As discussed in the Examples \ref{relation1/exp} and \ref{relation2/exp} below, both the relations hold for a continuous family of triangles, forming a curve in the space of parameters
\begin{equation}\nonumber
\T = \{(\alpha_2,\alpha_3) \;|\; \alpha_2,\alpha_3 > 0 \text{ and } \alpha_2 + \alpha_3 < \pi\}\,,
\end{equation}
and almost all the triangles in that family are typical. Moreover, the relation presented in Example \ref{relation1/exp} has minimal length also with respect to the set generators $C(t_1)$, being not difficult to see that, apart from commutators, any extra relation holding in $T_\tau$ for a typical triangle $\tau$ must have length at least 5 in terms of conjugates of $t_1$.

\begin{example}\label{relation1/exp}
Consider the stable word of length 18
\begin{equation}\nonumber
(r_1r_2r_3r_2r_3r_1r_2r_1r_3)^2
= \(t_1\)r_1r_3r_2r_1 \cdot t_1 \cdot \(t_1\)r_3r_1r_3r_2r_3 \cdot \(t_1\)r_1r_3 \cdot \(t_1\)r_3r_1r_3\,.
\end{equation}
The corresponding translation vector is 
\begin{equation}\label{rel2/eqn}
\begin{array}{l}\vrule width0pt depth0pt height12pt
t_1  + \(t_1\)r_1r_3 + \(t_1\)r_1r_3r_2r_1 + \(t_1\)r_3r_1r_3 + \(t_1\)r_3r_1r_3r_2r_3\\[2pt]
\kern20pt = t_1 + \(t_1\)r_1r_3 + \(t_1\)r_2r_3 + \(t_1\)r_1r_2r_1r_3 + \(t_1\)r_1r_3r_1r_3\,,\\[2pt]
\end{array}
\end{equation}
where the semplification is based on the equation (\ref{conjugate/eqn}), the commutativity of the rotations $s_is_j$ and $s_ks_l$, and the identities $\(t_1\)r_1r_2r_3=t_1$ and $s_i^2 = 1$. 
Now,  the oriented angles from $t_1$ to the five vectors in (\ref{rel2/eqn}) are respectively given by $0, 2(\alpha_2 + \alpha_3), 2\alpha_2, -2(\alpha_2 - \alpha_3)$ and $4\alpha_2$, and by replacing in (\ref{sys/eqn}) we get the system
\begin{equation}\nonumber
\left\{
\begin{array}{l}
\cos2(\alpha_2 + \alpha_3) + \cos2\alpha_2 + \cos2(\alpha_2 - \alpha_3) + \cos4\alpha_2 = -1\\[2pt]
\sin2(\alpha_2 + \alpha_3) + \sin2\alpha_2 - \sin2(\alpha_2 - \alpha_3) + \sin4\alpha_2 = 0
\end{array}
\right..
\end{equation}
By standard trigonometric identities, this system is equivalent to
\begin{equation}\nonumber
\left\{
\begin{array}{l}
(1 + 2\cos 2\alpha_2 + 2 \cos 2\alpha_3)\cos 2\alpha_2 = 0\\[2pt]
(1 + 2\cos 2\alpha_2 + 2 \cos 2\alpha_3)\sin 2\alpha_2 = 0
\end{array}
\right.,
\end{equation}
hence to the equation
\begin{equation}\nonumber
1 + 2\cos 2\alpha_2 + 2 \cos 2\alpha_3 = 0\,.
\end{equation}

The curve solutions of this equation in the space of parameters $\T$ is plotted on the left side of Figure \ref{figure4/fig}. Since non-typical triangles form a dense countable union of straight lines in $\T$, it is clear that only countably many triangles along the curve are non-typical. A very special case is represented by the triangle $d$ in the figure, whose angles are all rational multiples of $\pi$.
Hence, we can conclude that the considered word is a relation in $T_\tau$ for uncountably many typical non-generic triangles $\tau$, which form a dense subset of the curve. A sample of them is given by the five triangles $a,b,c,e,f$ depicted in the figure.
\end{example}

\begin{Figure}[htb]{figure4/fig}
\fig{}{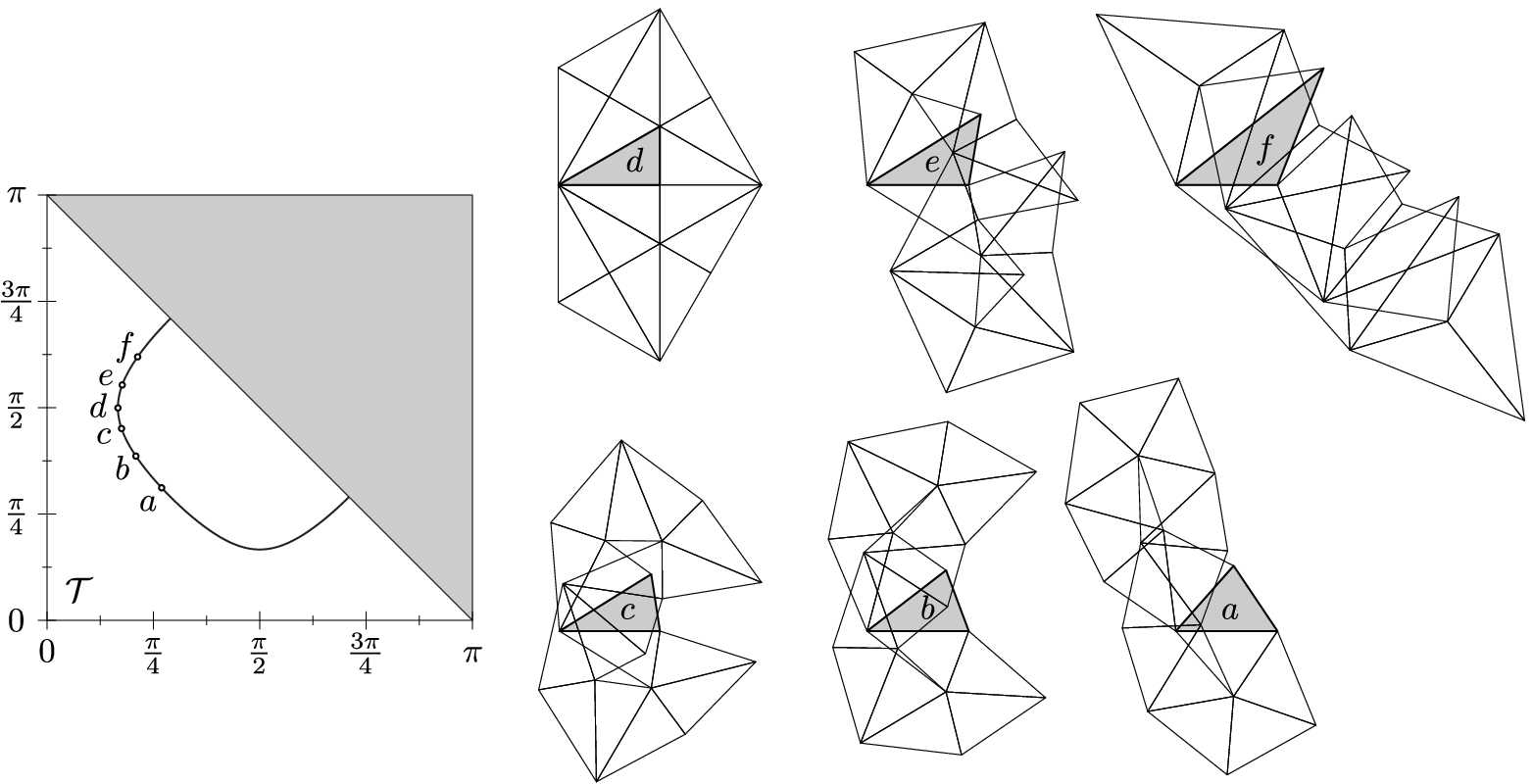}
    {}{The non-generic relation of Example \ref{relation1/exp}}
\end{Figure}

\begin{example}\label{relation2/exp}
Arguing as above, we see that the stable word of length 18
\begin{equation}\nonumber\kern-5pt
\begin{array}{l}\vrule width0pt depth0pt height12pt
(r_1r_2r_3r_2r_3r_1r_3r_1r_2)^2\\[2pt]
\kern20pt 
= \(t_1\)r_1r_3r_2r_1 \cdot t_1 \cdot \(t_1\)r_1r_3r_2r_3 \cdot \(t_1\)r_3r_1r_3r_2r_3 \cdot \(t_1\)r_1r_3 \cdot \(t_1\)r_3 \cdot \(t_1\)r_1r_3r_1r_2
\end{array}
\end{equation}
corresponds to the translation vector
\begin{equation}\nonumber
\begin{array}{l}\vrule width0pt depth0pt height12pt
t_1 + \(t_1\)r_3 + \(t_1\)r_1r_3 + \(t_1\)r_1r_3r_2r_1 + \(t_1\)r_1r_3r_1r_2 + \(t_1\)r_1r_3r_2r_3 + \(t_1\)r_3r_1r_3r_2r_3\\[2pt]
\kern20pt = t_1 + \(t_1\)r_1r_2 + \(t_1\)r_1r_3 + \(t_1\)r_2r_3 + \(t_1\)r_1r_3r_1r_2 + \(t_1\)r_1r_3r_2r_3 + \(t_1\)r_1r_3r_1r_3\,.
\end{array}\kern-5pt
\end{equation}
This leads to the system
\begin{equation}\nonumber
\left\{
\begin{array}{l}
(1 + 2\cos 2\alpha_2 + 2 \cos 2\alpha_3 + 2\cos2(\alpha_2+\alpha_3))\cos 2\alpha_2 = 0\\[2pt]
(1 + 2\cos 2\alpha_2 + 2 \cos 2\alpha_3 + 2\cos2(\alpha_2+\alpha_3))\sin 2\alpha_2 = 0
\end{array}
\right.,
\end{equation}
hence to the equation
\begin{equation}\nonumber
1 + 2\cos 2\alpha_2 + 2 \cos 2\alpha_3 + 2\cos2(\alpha_2 + \alpha_3) = 0\,.
\end{equation}
Hence, also in this case we can conclude that the considered word is a relation in $T_\tau$ for uncountably many typical non-generic triangles $\tau$, which form a dense subset of the curve represented by the equation.
\end{example}

Besides the two relations of length 18 given in the previous examples, our computer search also detected other non-generic relations holding for all the triangles along a curve in the parameter space $\T$, hence for uncountably many typical triangles. Namely, we there are 6 such relations of length 22 and 5 of length 24, but none of length 20.

Moreover, we found a certain number of non-generic relation holding only for isolated typical triangles. One of such relations is discussed in Example \ref{relation3/exp}.

Actually, systematic search produced even shorter relations holding in isolated triangles, which present strong evidence of being typical. But we were not able to prove that such triangles are really typical. The shortest one has length 22, and it is the unique that length, up to permutation of indices, inversion, conjugation and commutation of stable words. Up to the same equivalence, there are also 20 similar relations of length 24, some of which have the minimal length 5 with respect to $C(t_1)$. Such further relations in conjecturally typical triangles are illustrated by Example \ref{relation4/exp}.

\begin{example}\label{relation3/exp}
Consider the stable word of length 32
\begin{equation}\nonumber
\begin{array}{l}\vrule width0pt depth0pt height12pt
r_1r_3r_1r_2r_3r_2r_3r_1r_2r_3r_2r_3r_1r_2r_3r_1r_2r_3r_1r_2r_3r_2r_3r_1r_3r_1r_2r_1r_2r_3r_1r_3\\[2pt]
\kern20pt = \(t_1\)r_3r_1 \cdot \(t_1\)r_2r_1 \cdot \(t_1^2\)r_3r_1 \cdot \(t_1\)r_2r_1 \cdot \(t_1\)r_1r_3r_2r_1 \cdot t_1 \cdot \(t_1\)r_1r_3\,,
\end{array}
\end{equation}
whose corresponding translation vector is
\begin{equation}\nonumber
t_1 + \(t_1\)r_1r_3 + 2\(t_1\)r_2r_1 + 3\(t_1\)r_3r_1 + \(t_1\)r_2r_3\,.
\end{equation}
Proceeding as in the previous examples, we see that this word represents the identity in $T_\tau$ if and only if the angles $\alpha_2$ and $\alpha_3$ satisfy the system
\begin{equation}\nonumber
\left\{
\begin{array}{l}
4\cos 2\alpha_2 + 2 \cos 2\alpha_3 + \cos 2(\alpha_2+\alpha_3)= -1\\[2pt]
2\cos (\alpha_2+\alpha_3) \left(\sin (\alpha_2 + \alpha_3) - 2\sin(\alpha_2 -\alpha_3)\strut\right) = 0
\end{array}
\right..
\end{equation}\nonumber
\end{example}
Apart from the rational (mod $\pi$) solution $\alpha_2 = \alpha_3 = \pi/4$, the only other acceptable solution is
\begin{equation}\nonumber
\left\{\vrule height20pt width0pt\right.\!
\begin{array}{l}
\alpha_2 = \arctan \sqrt2\\[2pt]
\alpha_3 = \arctan \displaystyle\frac{\sqrt2}{3}
\end{array}\,.
\end{equation}
According to Theorem 2 of \cite{CoRaSa99}, this solution can be written in the form
\begin{equation}\nonumber
\left\{
\begin{array}{l}
\alpha_2 = q\pi \pm \langle3\rangle_2\\[2pt]
\alpha_3 = q'\pi \pm \langle11\rangle_2
\end{array}\right.,
\end{equation}
for certain rational numbers $q$ and $q'$, and certain angles $\langle 3\rangle_2$ and $\langle 11\rangle_2$ that are rationally independent together with $\pi$. This implies that the triangle is typical. The chains of reflections realizing the relation for such triangle is shown in Figure \ref{figure5/fig}.

\begin{Figure}[htb]{figure5/fig}
\fig{}{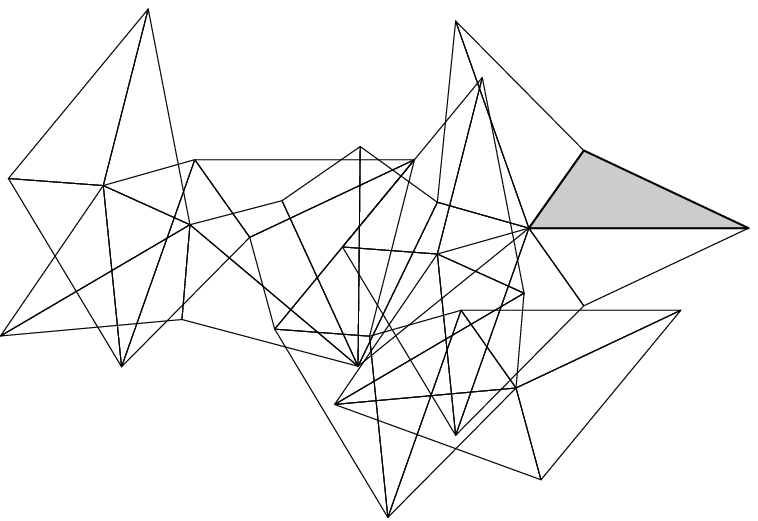}
    {}{The non-generic relation of Example \ref{relation3/exp}}
\end{Figure}

\begin{example}\label{relation4/exp}
The stable word of length 22
\begin{equation}\nonumber
\begin{array}{l}\vrule width0pt depth0pt height12pt
r_1r_2r_1r_2r_3r_1r_3r_1r_3r_2r_3r_1r_2r_3r_2r_3r_2r_3r_2r_3r_1r_2\\[2pt]
\kern20pt = \(t_1\)r_2r_1 \cdot \(t_1\)r_1r_3r_2r_1 \cdot t_1 \cdot \(t_1\)r_1r_3 \cdot \(t_1\)r_2r_3r_1r_3 \cdot \(t_1\)r_2r_3r_2r_3r_1r_3 \cdot \(t_1\)r_3
\end{array}
\end{equation}
represents the identity in $T_\tau$ if and only if
\begin{equation}\nonumber
\left\{
\begin{array}{l}
\cos 2\alpha_2 + 2 \cos 2\alpha_3 + \cos 2(\alpha_2+\alpha_3) + \cos 2(2\alpha_2+\alpha_3) \cos 2(3\alpha_2+2\alpha_3) = -1\\[2pt]
\sin 2\alpha_2 + \sin 2(\alpha_2+\alpha_3) + \sin 2(2\alpha_2+\alpha_3) \sin 2(3\alpha_2+2\alpha_3) = 0
\end{array}
\right..
\end{equation}

\noindent
The only two acceptable approximate solutions of the system are
\begin{equation}\nonumber
\left\{
\begin{array}{l}
\alpha_2 = 0.3675592642\,\pi\\[2pt]
\alpha_3 = 0.1932064551\,\pi
\end{array}
\right.
\ \ \text{and} \ \ 
\left\{
\begin{array}{l}
\alpha_2 = 0.5971477967\,\pi\\[2pt]
\alpha_3 = 0.2299624978\,\pi
\end{array}
\right..
\end{equation}
Analogously, the stable word of length 24
\begin{equation}\nonumber
\begin{array}{l}\vrule width0pt depth0pt height12pt
r_1r_2r_1r_2r_1r_3r_1r_3r_2r_1r_3r_1r_2r_1r_3r_2r_3r_2r_3r_1r_2r_3r_2r_3\\[2pt]
\kern20pt = \(t_1^{-1}\)r_3r_1 \cdot \(t_1^{-1}\)r_2r_1r_3r_1 \cdot \(t_1^{-1}\)r_3r_1r_2r_1r_3r_1 \cdot t_1 \cdot \(t_1\)r_2r_3
\end{array}
\end{equation}
represents the identity in $T_\tau$ if and only if
\begin{equation}\nonumber
\left\{
\begin{array}{l}
\cos 2\alpha_2 - \cos 2(\alpha_2+\alpha_3) + \cos 2(\alpha_2-\alpha_3) + \cos 2(2\alpha_2-\alpha_3) = 1\\[2pt]
\sin 2\alpha_2 + \sin 2(\alpha_2+\alpha_3) + \sin 2(\alpha_2-\alpha_3) \sin 2(2\alpha_2-\alpha_3) = 0
\end{array}
\right..
\end{equation}
The only acceptable approximate solution of the system is
\begin{equation}\nonumber
\left\{
\begin{array}{l}
\alpha_2 = 0.2961623095\,\pi\\[2pt]
\alpha_3 = 0.4392394514\,\pi
\end{array}
\right..
\end{equation}
The chains of reflections realizing both the non-generic relations above are shown in Figure \ref{figure6/fig}, on the left side for the two triangles where the former relation holds and on the right side for the unique triangle where the latter holds.
\end{example}

\begin{Figure}[htb]{figure6/fig}
\fig{}{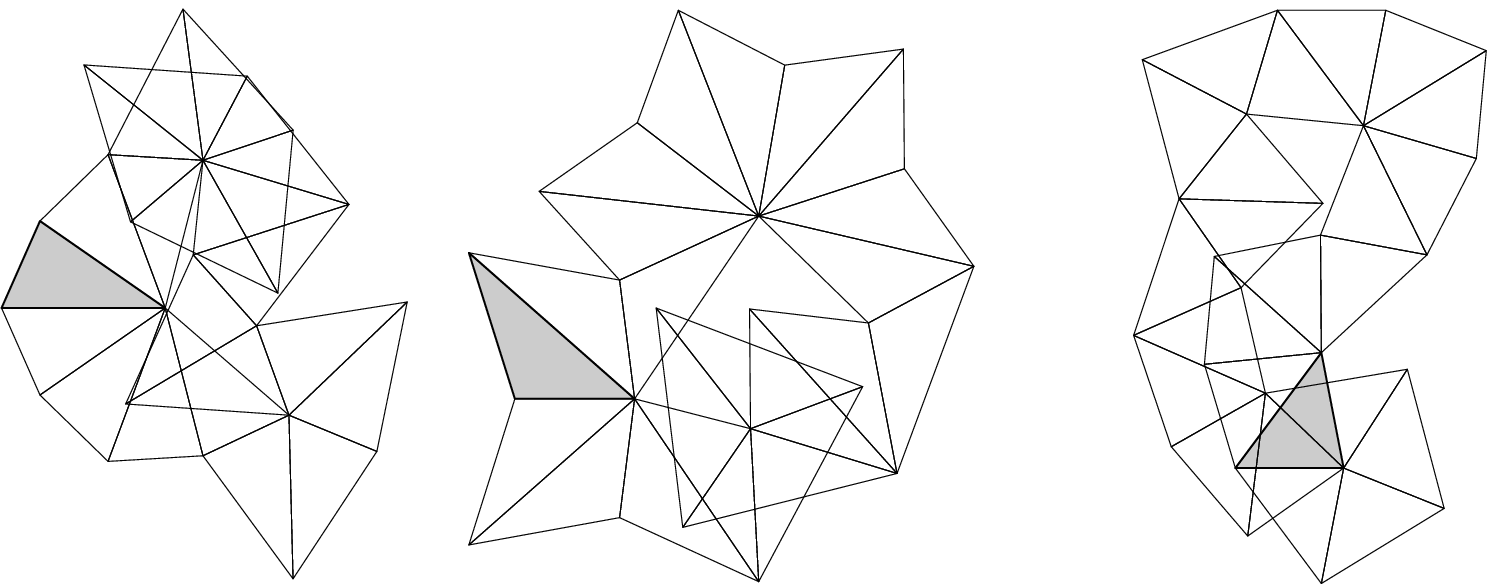}
    {}{The non-generic relations of Example \ref{relation4/exp}}
\end{Figure}



\end{document}